\documentclass[11pt, oneside]{amsart}
\usepackage[utf8]{inputenc}
\usepackage[T1]{fontenc}

\usepackage{amsmath}
\usepackage{mathtools}
\usepackage{mathrsfs}
\usepackage{amsthm}
\usepackage{algorithmicx}

\usepackage{graphicx}
\usepackage{geometry}
\usepackage{amssymb}
\usepackage{xcolor, soul}
\usepackage{stmaryrd}
\usepackage{svg}
\usepackage{thmtools}

\newcommand{\eigenvect}{v}
\newcommand{\eigenvectw}{w}

\renewcommand{\geq}{\geqslant}
\renewcommand{\leq}{\leqslant}

\newcommand{\spec}{\operatorname{Spec}}
\newcommand{\sing}{\operatorname{Sing}}
\newcommand{\GL}{\operatorname{GL}}
\newcommand{\Int}{\operatorname{int}}

\newcommand{\End}{\operatorname{End}}
\newcommand{\Aut}{\operatorname{Aut}}
\newcommand{\CLip}{\operatorname{CLip}}

\newcommand{\arcsinh}{\operatorname{arcsinh}}
\newcommand{\R}{\mathbb{R}}
\newcommand{\N}{\mathbb{N}}
\newcommand{\C}{\mathbb{C}}
\newcommand{\mytop}{\operatorname{top}}
\newcommand{\strategyMax}{\mathcal{B}^s}
\newcommand{\strategyMin}{\mathcal{A}^s}
\newcommand{\payoff}{J}
\newcommand{\Funk}{\operatorname{Funk}}
\newcommand{\RFunk}{\operatorname{RFunk}}
\newcommand{\Th}{\operatorname{Th}}
\newcommand{\Hil}{\operatorname{Hil}}
\newcommand{\A}{\mathcal{A}}
\newcommand{\B}{\mathcal{B}}
\usepackage{xspace}
\newcommand{\Min}{{\sc Min}\xspace}
\newcommand{\Max}{{\sc Max}\xspace}
\def\<#1,#2>{\langle #1,#2\rangle}
\usepackage{hyperref}
\hypersetup{
    colorlinks,
    linkcolor={red!50!black},
    citecolor={blue!50!black},
    urlcolor={blue!80!black}
}

\DeclareMathOperator*{\argmax}{arg\,max}
\DeclareMathOperator*{\argmin}{arg\,min}
\DeclareMathOperator*{\co}{co}

\newcommand{\Lip}{\operatorname{Lip}_1}
\usepackage{cleveref}
\newtheorem{Theorem}{Theorem}
\newtheorem{Corollary}[Theorem]{Corollary}
\newtheorem{Proposition}[Theorem]{Proposition}
\newtheorem{Lemma}[Theorem]{Lemma}
\newtheorem{Definition}[Theorem]{Definition}
\newtheorem{Assumption}[Theorem]{Assumption}
\newtheorem{Example}[Theorem]{Example}
\newtheorem{Remark}[Theorem]{Remark}

\usepackage[sorting=none,natbib=true,backend=biber,doi=false,url=false,
isbn=false,eprint=true,style= alphabetic, giveninits=true,maxbibnames=99]{biblatex}
\usepackage{todonotes}

\title{The competitive spectral radius of families of nonexpansive mappings}

\author{Marianne Akian \and Stéphane Gaubert \and Loïc Marchesini}
\email{Firstname.Lastname@inria.fr}
\address{INRIA and CMAP, \'Ecole polytechnique, Institut polytechnique de Paris, CNRS. Address: CMAP, \'Ecole polytechnique, 91128 Palaiseau C\'edex, France}
\date{October 23$^{\text{rd}}$ 2024}
\keywords{Nonexpansive mappings, joint spectral radius, Hilbert and Thompson metrics, zero-sum games, uniform value}

\begin{document}

\begin{abstract}
  We consider a new class of repeated zero-sum games in which the payoff is the escape rate of a switched dynamical system, where at every stage, the transition is given by a nonexpansive operator depending on the actions of both players. This generalizes to the two-player (and non-linear) case the notion of joint spectral radius of a family of matrices.
  We show that the value of this game does exist, and we characterize it in terms of an infinite dimensional non-linear eigenproblem.
  This provides a two-player analogue of Ma\~{n}e's lemma from ergodic control and
    extends to the two-player case results of Kohlberg and Neyman (1981),
    Karlsson (2001), and Vigeral and Gaubert (2012), concerning the asymptotic
    behavior of nonexpansive mappings.
    We discuss two special cases of this game: order preserving and positively homogeneous self-maps of a cone equipped with Funk's and Thompson's metrics, and translations of a finite dimensional normed space.
  \end{abstract}

\maketitle

{ \hypersetup{linkcolor=black}
\tableofcontents
}

\section{Introduction}
\subsection{Context}
A self map $T$ of a metric space $(X,d)$ is \emph{nonexpansive} if 
$d(Tx,Ty)\leq d(x,y)$ for all $x,y\in X$.
Then, its  \emph{escape rate} is defined by
\[
\rho(T) = \lim_{k \to \infty}\frac{d(T^k\bar{x}, \bar{x})}k \enspace ,
\]
where $\bar{x}$ is an arbitrary element of $X$.
Since $T$ is nonexpansive, the existence of the limit follows from a subadditive
argument, and the limit is independent of the choice of $\bar{x}$. Kohlberg and Neyman proved the following characterization of $\rho(T)$ when $T$ acts on a subset of a normed space.

\begin{Theorem}[\cite{KoNe}]\label{KoNey}
	Let $X$ be a star-shaped subset of a normed space $E$ and $T: X \to X$ be a nonexpansive map, then there exists a continuous linear form $f$ on $E$  of dual norm one such that, for all $x \in X$,
	\[
        \rho(T) =  \lim_{k \to \infty}  \frac{\lVert T^k x \rVert}{k}
        = \inf_{y \in X}\lVert Ty - y \rVert 
        = \lim_{k \to \infty}\<f,\frac{T^k x}{k}> \enspace .
	\]
Moreover, $f$ can be chosen in the weak-$*$ closure of the set of extreme 
points of the dual unit ball.
\end{Theorem}
They also showed that for all $x\in X$ it is possible to choose
$f$ in \Cref{KoNey} such that for all $k \in \N^*$, we have
\[
\<f,T^k x> - \<f,x> \geq k \rho(T)\enspace .
\]
Karlsson extended this result by proving the following Denjoy-Wolff type theorem, which relies on Gromov's compactification
of a metric space by horofunctions~\cite{Gromov1981}.
The latter are particular Lipschitz functions of Lipschitz constant $1$ which arise when compactifying the space $X$ as follows: we identify a point $x\in X$ to the function $i(x)\coloneqq d(x,\bar{x}) - d(x,\cdot)$ from $X$ to $\R$, where $\bar{x}\in X$ is an arbitrary basepoint; then, we define the horocompactification of $X$ as the closure of $i(X)$ in the topology of pointwise convergence.
(This choice of topology was made e.g.\ in~\cite{rieffel,GAUBERT_2011,Karlsson2024}; the topology of uniform convergence on bounded sets
was originally considered~\cite{Gromov1981},
leading to a different class of horofunctions. Note also that the opposite convention of sign, considering $ - d(x,\bar{x}) + d(x,\cdot)$ instead of $i(x)$ above, is often made.)
The functions of the closure $\overline{i(X)}$ are called {\em metric functionals}~\cite{Karlsson2024}.
The functions of $\overline{i(X)}/i(X)$ are called \emph{horofunctions}, they constitute the
\emph{horoboundary}.

\begin{Theorem}[See~\cite{Ka}]\label{th-karlsson}
  Let $X$ be a proper metric space, $T$ a nonexpansive self-map of $X$,
  and suppose that every orbit of $T$ in $X$ is unbounded. Then, 
for all $x\in X$, there is a horofunction $h$ 
such that 
\[
h(T^k x)-h(x) \geq k \rho(T)   \quad \forall k\geq 0\enspace .
	\]
\end{Theorem}
Gaubert and Vigeral refined \Cref{th-karlsson}, showing that the horofunction $h$ can be chosen to be independent of $x\in X$ when
$X$ satisfies an assumption of non-positive curvature in the sense of Busemann.
They consider more generally maps $T$ which are nonexpansive with regard to a \emph{hemi-metric} of $X$, i.e.\ a map $d$ taking values in $\R$, verifying the triangular inequality and a weak separation axiom. They established the following maximin characterization of the escape rate.

\begin{Theorem}[See~\protect{\cite[Theorem~1]{GAUBERT_2011}}]\label{GaubertVigeral}
	Let $T$ be a nonexpansive self-map of a complete metrically star-shaped hemi metric space $(X,d)$. Then
	\[
	\rho(T) = \inf_{y \in X} d(Ty,y) = \max_{h \in \mathcal{H}}\inf_{x \in X}(h(Tx)-h(x)) \enspace ,
	\]
	where $\mathcal{H}$ is the set of metric functionals. Moreover if $d$ is bounded from below and if $\rho(T) > 0$ then any function attaining the maximum is a horofunction. 
\end{Theorem}
In particular, any horofunction $h$ attaining the above maximum verifies
$h(T^kx) - h(x) \geq k\rho(T)$ for all $k \in \N^*$, $x \in X$.
Recently, Karlsson refined Theorem~\ref{GaubertVigeral}, showing that when $T$ is an isometry,
there exists a metric functional $h$ such that $h(Tx) - h(x) = \rho(T)$
for all $x\in X$, see~\cite{Karlsson2024}.
Moreover, Lemmens, Lins, Nussbaum and Wortel showed that the nonpositive curvature assumption
implied by the metrically star shaped condition in Theorem~\ref{GaubertVigeral} can be dispensed
with if $T$ is an order-preserving and positively homogeneous self-map of the interior of a normal
cone in a Banach space, satisfying appropriate conditions, and $d$ is the Funk metric,
see~\cite[Theorem~4.6 and Corollary~6.4]{Lemmens2018}.

\subsection{Main results}
In the present paper, we introduce a two-player game extension of the notion of escape rate.
We consider a family of nonexpansive self-maps
$(T_{ab})_{(a,b) \in \mathcal{A}\times \mathcal{B}}$
of a hemi-metric space $(X,d)$. We refer to the two players as \Min and \Max.
The sets $\mathcal{A}$ and $\mathcal{B}$ are interpreted as
the {\em action spaces} of \Min and \Max, respectively.
We assume that \Min and \Max alternatively select actions
$a_1\in \mathcal{A}, b_1\in \mathcal{B}, a_2\in \mathcal{A}, b_2\in \mathcal{B},\dots$,
each observing the previous actions of the other player before selecting their next action.

The \emph{escape rate} induced by this infinite sequence of actions is
defined as follows
\[
	\limsup_{k \to \infty}\frac{d(T_{a_k b_k} \circ \dots \circ T_{a_1 b_1}(\bar{x}), \bar{x})}{k} \enspace ,
\]
where as above, the limit is independent of the choice of $\bar{x}$.
Player \Max aims to maximize this quantity, while \Min aims to minimize it.
We refer to this two-player zero-sum game as the \emph{Escape Rate Game}.
We always suppose that the action spaces are compact, and make
a technical assumption 
which is satisfied in particular if the map $(a,b,x)\mapsto
T_{ab}(x)$ is continuous (see \Cref{a-1} below). 
We are interested in the existence of \emph{uniform value},
a notion studied by Mertens and Neyman~\cite{mertens_neyman}.
Informally, a zero-sum repeated game is said to have a \emph{uniform value} $\rho$ if both players can guarantee $\rho$ up to an arbitrary precision provided that they play for a sufficiently long time (see \Cref{def-uniform} below).

To solve this game, we introduce the following \emph{Shapley operator} $S$, acting on the space  $\Lip$
of Lipschitz functions of constant $1$ from $X$ to $(\R,\delta_1)$ with
$\delta_1(x,y)=x-y$:
\[ 
Sv(x)\coloneqq \inf_{a \in \mathcal{A}}\sup_{b \in \mathcal{B}}v(T_{ab}(x)) \enspace .
\]
Our first main result is the following maximin characterization of the value
of the escape rate game.

\begin{Theorem}\label{main}
	The escape Rate Game has a uniform value $\rho$ given by
	\begin{align*}
	\rho = \lim_{k\to\infty}
        \frac{[S^k d(\cdot,\bar{x})](\bar x)}{k} &= \max\{ \lambda \in \R \mid \exists\; v \in \Lip,\; \lambda + v \leq Sv \}\\
        &= \max_{v \in \Lip}\inf_{x \in X}(Sv(x)-v(x)) \enspace .
	\end{align*}
\end{Theorem}
\noindent We show in particular that player \Max
has an optimal stationary strategy, and that player \Min
has an optimal (non stationary) strategy.

We say that a function $f\in\Lip$ is \emph{distance-like} if
there exists a constant $\alpha$ such that 
$f(x) \geq d(x,\bar x) + \alpha$.
We denote by $\mathscr{D}$ the set of distance like functions.
We say that a nonexpansive self map $T$ of a hemi metric space $(X,d)$
is an \emph{almost-isometry} with constant $\gamma\in \R$ if it satisfies 
$d(T(x),T(y))\geq d(x,y)-\gamma$, for all $x,y\in X$. Having $S$ act on $\mathscr{D}$, we get the following dual minimax characterization of $\rho$.
\begin{Theorem}\label{th-dual}
  Suppose that the maps $T_{ab}$ constitute a family of almost-isometries
  with a uniform constant. Then. the value of the escape rate game
  satisfies
\[ 
\rho = \inf\{\lambda \in \R \mid v \in \mathscr{D},\; \lambda + v \geq Sv  \} = \inf_{v \in\mathscr{D}}\sup_{x \in X}(Sv(x)-v(x)) \enspace .
\]\end{Theorem}
The infimum is not achieved in general. In contrast, in the maximin characterization, the supremum is always achieved, and the almost-isometry assumption is not needed.

The above minimax and maximin characterizations
of the escape rate are reminiscent of ``Collatz-Wielandt'' type characterizations of the spectral radius,
established by Nussbaum in the context of non-linear Perron-Frobenius theory,
see~\cite{nussbaum86} and also~\cite{nussbaumlemmens,Lemmens2018}.
They may also be thought of as two player versions of the ``Mañé-Conze-Guivarch lemma'' characterizing the value of an ergodic control problem, 
see~\cite{Mane1992,Ma1996} and the discussion by Bousch in~\cite{Bousch2011}.
The functions $v$ arising in these results are also closely related with the weak-KAM solutions defined by Fathi~\cite{fathi97b} for continuous time problems, see also~\cite{Fathi2004}.

Our results apply in particular to the following example of hemi-metric. Suppose that $C\subset \R^n$ is a
proper cone (i.e., convex, closed, pointed, of non-empty interior), and let $\leq_C$ denote
the induced partial order on $\R^n$, so that $x\leq_C y$ if $y-x\in C$. The interior
of $C$ can be equipped with the \emph{Funk} hemi-metric, defined by
\begin{align}
  \Funk(x,y) = \log\inf \{ \alpha>0 \mid x\leq \alpha y\} \enspace. 
  \label{e-def-funk}
\end{align}
This hemi-metric is a key ingredient in Hilbert geometry~\cite{Funk1929,Papadopoulos2008}.
A self-map $T$ of the interior of $C$ is {\em order preserving} if $x\leq_C y$ implies $T(x)\leq_C T(y)$, for all $x,y$ in the interior
of $C$, and it is {\em positively homogeneous of degree $1$} if $T(\alpha x)= \alpha T(x)$
holds for all $\alpha>0$. These two properties are equivalent to $T$ being nonexpansive in the Funk hemi-metric.
This allows us to define escape rate games for dynamics that are order preserving and positively homogeneous of degree $1$.
Moreover, Corollary~\ref{coro-cones} below shows that for this special class of games, the condition that the maps $T_{ab}$ be almost-isometries
in Theorem~\ref{th-dual}
can be dispensed with; to get a minimax characterization of the value,
it suffices that each map $T_{ab}$ extend
continuously to the closed cone $C$.

By specializing this approach to linear maps, we can solve \emph{matrix multiplication games},
defined by Asarin et al.\ in \cite{asarin_entropy}. In these games,
one considers two compact sets of matrices $\mathcal{A}$ and $\mathcal{B}$
with entries in $\mathbb{C}$,
and the players \Min and \Max alternatively select matrices
$a_1\in\mathcal{A}$, $b_1\in\mathcal{B}$, $a_2\in\mathcal{A}$, \dots
Player \Max wants to maximize the payoff
\begin{align}\label{e-payoff}
\limsup_{k} \|a_1 b_1 \dots a_k b_k\|^{1/k} \enspace ,
\end{align}
whereas Player \Min wants to minimize it. This payoff is independent
of the choice of the norm $\|\cdot\|$. 

When $\mathcal{A}$ and $\mathcal{B}$ are included in $\operatorname{GL}_n$,
the matrix multiplication game is a special case of escape rate game.
Indeed, we consider the cone of Hermitian positive semidefinite matrices,
$C=H_n^{+}$. Then, the order $\leq_C$ is nothing but the {\em Loewner order},
and $\Funk(x,y) = \log \lambda_{\max}(xy^{-1})$, where $\lambda_{\max}$
denotes the maximal eigenvalue of a matrix
with real spectrum.
We associate to a pair of matrices $(a,b)\in \mathcal{A}\times \mathcal{B}$
the self map of the interior of $H_n^+$ defined by $T_{ab}(x) = b^*a^* x a b$.
Observe that $T_{ab}$ is an isometry in the Funk metric.
Then, choosing $\|\cdot\|$ to be the spectral norm,
we have
\[
\log \|a_1b_1 \dots a_k b_k\| = \frac{1}{2}
\Funk(T_{a_k b_k}\circ \dots \circ T_{a_1 b_1}(I), I) \enspace ,
\]
where $I$ is the identity matrix.
By applying \Cref{main} and \Cref{th-dual}, we obtain the existence together with characterizations of the uniform value of matrix multiplication games with matrices in $\GL_n$. Similarly, \Cref{main} and Corollary~\ref{coro-cones} allow us to derive the same
conclusion for matrices preserving the interior of a proper cone. In the case of nonnegative matrices,
the existence of the value was previously known only under a restrictive ``rectangularity condition''~\cite{asarin_entropy}.

In the special case of ``minimizer free'' matrix multiplication games,
in which Player \Min is a dummy
so that $\mathcal{A}=\{I\}$, the value of the escape rate
game coincides with the {\em joint spectral radius} introduced
by Rota and Strang~\cite{Rota1960}. When Player \Max is a dummy,
it coincides with the notion of {\em lower spectral radius} or {\em joint spectral subradius}
introduced much later by Gurvits~\cite{Gurvits1995},
and further studied by Jungers, Bochi and Morris~\cite{Jungers_2009,Bochi2014}.
In particular, we show in \Cref{non-defective}, in the \emph{minimizer free} case, that the existence of a distance-like function reaching the infimum in \Cref{th-dual} is equivalent to the non-defectivity of the family of matrices.

Because of these special cases, we call {\em the competitive spectral radius} the value
of an escape rate game. It also includes as special one-player cases
the notion of joint spectral radius of families of isometries
introduced recently by Breuillard and Fujiwara~\cite{Breuillard2021}, as well as
the joint spectral radius of family of topical
maps studied by Bousch and Mairesse~\cite{Bousch_2001}.
We also note that Kozyakin introduced in~\cite{Kozyakin_2019} a notion of minimax joint spectral radius, which differs from ours (instead of alternating moves, one of the players first chooses an infinite sequence of matrices, leading to a Stackelberg game).

Concrete examples of competitive spectral radii arise
in the study of topological entropies of transition systems~\cite{asarin_entropy}
and in population dynamics, in particular in problems from mathematical biology
and epidemiology, see in particular~\cite{billy,calvezgabriel,Chakrabarti2008}. They also arise in robust Markov decision processes, in which the second player
represents the uncertainty, see~\cite{Goyal2023}.

We illustrate our results by solving, in Section~\ref{sec-vecadd},
a special class of escape rate game, in which the nonexpansive dynamics consists
of translations acting on the Euclidean space. We give an explicit formula
for a maximizing function $v$ in the maximin characterization of \Cref{th-dual},
which is based on a result of discrete geometry (Shapley-Folkman lemma).

\section{Escape rate games: definition and first properties}

\subsection{Hemi-metric spaces}
We shall need the following notion of asymmetric metric \cite{GAUBERT_2011}.
\begin{Definition}[Hemi-metric]\label{def-hemi}
	A map $d:X\times X \to \R$ is called a \emph{hemi-metric} on the set $X$ if it satisfies the two following conditions:
	\begin{enumerate}
		\item $\forall(x,y,z) \in X^3, \; d(x,z) \leq d(x,y) + d(y,z)$ (triangular inequality)
		\item $\forall (x,y) \in X^2, \; d(x,y) = d(y,x) = 0$ if and only if $x=y$
	\end{enumerate}
\end{Definition}

A hemi-metric may take negative values.
A basic example of hemi-metric on $\R^n$ is the map $\delta_n(x,y)\coloneqq \max_{i\in[n]} (x_i-y_i)$.

Given a hemi-metric $d$, the map
\[ d^{\circ}(x,y) = \max(d(x,y), d(y,x)) \enspace,
\]
is always an ordinary metric on $X$. 
We will refer to it as the \emph{symmetrized metric} of $d$.
In the sequel, we equip $X$ with the topology induced by $d^{\circ}$.

Given two hemi-metric spaces $(X,d)$ and $(Y,d')$, we introduce the notion of a
{\em $1$-Lipschitz} or {\em nonexpansive} function $\eigenvect$ with respect to $d$ and $d'$. Such a function verifies $d'(\eigenvect(x), \eigenvect(y)) \leq d(x,y)$ for all $x,y \in X$. We will especially be interested in the nonexpansive functions from $(X,d)$ to $(\R,\delta_1)$.
It is easy to verify that a nonexpansive function between two hemi-metric spaces is nonexpansive with respect to the symmetrized metrics generated by their respective hemi-metrics. In particular, such a function is continuous with respect to the topologies induced by these symmmetrized metrics.

\subsection{The Escape Rate Game}

The \emph{Escape Rate Game} is the following two-player deterministic perfect information game. 

We fix two non-empty compact sets $\mathcal{A}$ and $\mathcal{B}$, that will represent the action spaces of the two players. We also fix a hemi-metric space
$(X,d)$, which will play the role of the state space of the game. To each pair
$(a,b) \in \mathcal{A}\times \mathcal{B}$, we associate a nonexpansive
self-map $T_{ab}$ of $(X,d)$. The initial state $x_0 \in X$ is given. We construct inductively a sequence of states
$(x_k)_{k\geq 0}$ as follows. Both players observe the current state $x_k$.
At each turn $k\geq 1$, player \Min chooses an action $a_k \in \mathcal{A}$ and then, after having observed the action $a_k$,
player \Max chooses an action $b_k \in \mathcal{B}$.
The next state is given by $x_{k}= T_{a_{k} b_{k}}(x_{k-1})$.

First, we define the {\em game in horizon $k$}, denoted by $\Gamma_k$.
There are $k$ successive turns, and player \Min pays to \Max the following amount:
\begin{equation}
	\payoff_k(a_1 b_1 \dots a_k b_k) = d(T_{a_k b_k} \circ \dots \circ T_{a_1 b_1}(x_0),x_0) \enspace .\label{e-def-jk}
\end{equation}
Player \Max aims to maximize this quantity, whereas
player \Min aims to minimize it.

We also define the \emph{escape rate game}, denoted $\Gamma_\infty$. Here, an infinite number of turns are played, and player \Min aims to minimize the {\em escape rate}
\begin{equation}
\payoff_\infty(a_1b_1a_2b_2\dots)\coloneqq	\limsup_{k \to \infty}  \frac{\payoff_k (a_1 b_1 \dots a_k b_k)}k \enspace ,\label{e-limsup}
\end{equation} 
while player \Max wants to maximize it. It follows from
the nonexpansiveness of the maps $T_{ab}$ that this limit is independent of the choice of $x_0$.

A {\em strategy} is a map that assigns a move to any finite history of the game (sequence of states and actions
in successive turns).
To formalize this notion, it will be convenient to use
the notation $\mathcal{C}^*$
for the set of all finite sequences of elements from a set $\mathcal{C}$.
Then, a strategy for \Min is a map
\begin{equation*} 
	\sigma : (X\times \mathcal{A} \times \mathcal{B})^*\times X \to \mathcal{A} \enspace ,
\end{equation*}
where the action chosen at turn $k+1$ depends on the history of states and actions $(x_0,a_1,b_1,x_1,\ldots, a_k,b_k,x_k)$.
Similarly, a strategy for \Max is a map
\begin{equation*}
	\tau : (X\times \mathcal{A} \times \mathcal{B})^* \times X \times \mathcal{A} \to \mathcal{B} \enspace .
\end{equation*}
The set of strategies of \Min (resp \Max) will be written $\strategyMin$ (resp $\strategyMax$).
A strategy $\sigma \in \strategyMin$ of Min is called \emph{stationary} if it is independent of the turn $k$ and depends only on the position in the space,
i.e.\ if for all $x \in X$, the restriction of $\sigma$, $\sigma_x : (X \times \mathcal{A} \times \mathcal{B})^* \times \{x\} \to \mathcal{A}$ is constant.
Similarly, a strategy of \Max\ is \emph{stationary} if it is independent of the turn $k$, and depends only on
the current state and last action of \Min.
Since the game is deterministic, any strategy can be rewritten as a function of the initial state and of the history of actions,
but then the stationarity may not be visible.

For every pair $(\sigma,\tau)\in \strategyMin\times\strategyMax$, 
we denote by $\payoff_k(\sigma,\tau)$ the payoff of $\Gamma_k$,
as per \eqref{e-def-jk},
assuming that the sequence of actions $a_1,b_1,a_2,b_2,\dots$ is generated according to the strategies $\sigma$ and $\tau$, and similarly, we denote by
$\payoff_\infty(\sigma,\tau)$ the escape rate defined above.

Recall that the {\em value} of a game is the unique quantity that the players can both guarantee. Formally
a game with strategy spaces $\strategyMin$ and $\strategyMax$ and payoff
function $\strategyMin \times \strategyMax \to \R$, $(\sigma,\tau)\mapsto
\payoff(\sigma,\tau)$, {\em has a value} $\lambda \in \R$ if $\forall \epsilon > 0, \; \exists \sigma^*\in\strategyMin, \tau^*\in\strategyMax$ such that
\begin{equation*}
\payoff(\sigma^*, \tau) \leq \lambda + \epsilon \text{ and }\payoff(\sigma, \tau^*) \geq \lambda - \epsilon \enspace ,
\end{equation*}
for all strategies $\sigma\in\strategyMin$ and $\tau\in\strategyMax$. Such strategies are called $\epsilon$-optimal strategies. When $\epsilon=0$,
one gets \emph{optimal strategies}, and in particular, $\lambda= \payoff(\sigma^*, \tau^*)$. Equivalently, the game has a value if the
following \emph{minimax} and \emph{maximin} coincide:
\[
\inf_{\sigma\in\strategyMin}\sup_{\tau\in\strategyMax} \payoff(\sigma,\tau) = \sup_{\tau\in\strategyMax}\inf_{\sigma\in\strategyMin} \payoff(\sigma,\tau) \enspace.
\]
The outer infimum and supremum are reached when the players have optimal strategies.

\subsection{Examples of escape rate games}

\subsubsection{Vector Addition Games}
We consider the space $X =\R^n$ equipped with the metric $d$ induced
by a norm $\|\cdot\|$. The sets of actions are 
two non-empty compact subsets
$\mathcal{A}$ and $\mathcal{B}$ of $\R^n$.
Each pair of actions $(a,b)\in \mathcal{A}\times \mathcal{B}$
induces a translation  $T_{ab}$ acting on $\R^n$, given by $T_{ab}(x) = x+a+b$,
which is not only nonexpansive but also an isometry.
The state of the game evolves
according to the dynamics $x_{k}= a_{k}+b_{k} +x_{k-1}$.
Then, the escape rate is given by
\[
\payoff_\infty(a_1b_1a_2b_2\dots) = \limsup_{k\to\infty}\frac{\|a_1+b_2+ \dots +a_k+b_k\|}{k} \enspace.
\]

\subsubsection{Conical Escape Rate Games}
\label{sec-funk}
As in the introduction, we assume that $C\subset \R^n$ is a proper cone---i.e., a cone that is convex, closed, pointed, of non-empty interior.
For all $y,z\in C\setminus\{0\}$, we define 
\[
d(y,z) \coloneqq
\Funk(y,z)\coloneqq \log\inf \{\alpha>0 \mid y\leq \alpha z\}
\enspace .
\]
In particular, $\Funk(y,z)=\log\inf\emptyset=+\infty$ if there is no $\alpha>0$ such that
$y\leq \alpha z$.
A {\em part} of the cone $C$ is an equivalence class of $C\setminus\{0\}$ for the relation $y\sim z$ if $\alpha' z\leq y\leq \alpha z$ for
some $\alpha,\alpha'>0$. The map $\Funk(\cdot,\cdot)$ is referred to as the Funk metric.
Its restriction to every part is a hemi-metric in the sense of Definition~\ref{def-hemi}.

The symmetrized metric, known as {\em Thompson}'s part metric,
\[
d^\circ(y,z) = \Th(y,z)\coloneqq \max (\Funk(y,z),\Funk(z,y)) \enspace ,
\]
is a bona fide metric on every part of $C$.
A related ``metric'' of interest is {\em Hilbert}'s projective metric
\[
\Hil(y,z)\coloneqq \Funk(y,z)+ \Funk(z,y) \enspace .
\]
The term projective metric refers to the fact that $\Hil(y,z)=0$ iff $y$ and $z$
are proportional.
See~\cite{Nussbaum1988,Papadopoulos2008,nussbaumlemmens} for background on these metrics.

For instance, if $C=(\R_{\geq 0})^n$, then $\Int C=(\R_{>0})^n$, and
for all $y,z\in (\R_{>0})^n$, setting $\mytop(x)\coloneqq\max_i x_i$,
\[
\Funk(y,z)  = \log \max_{i \in [n]}\frac{y_i}{z_i} = \mytop(\log y-\log z),
\qquad
\Th(y,z) = \|\log y-\log z\|_\infty
\enspace .
\]
Similarly, if $C=H_n^+$ is the cone of hermitian positive semidefinite matrices, thought of as a real vector space,
then, $\Int C$ is the cone of positive definite hermitian matrices, and
\[
\Funk(y,z) = \log \lambda_{\max}(yz^{-1}) \enspace .
\]
\begin{Remark}
 Note that $\Funk(y,z)$ may take negative values.
 In contrast,
  Papadopoulos and Troyanov require in~\cite{Papadopoulos2008} weak metrics to be nonnegative,
 and so define the Funk metric to be $\max(0, \Funk(y,z))$. We work
  with the unsigned version as it captures more information.
\end{Remark}

We call \emph{order preserving} a self-map $f$ of $\Int C$ such that $x \leq_C y \implies f(x) \leq_C f(y)$. It is called \emph{positively homogeneous} (of degree $1$) if for all $\lambda > 0, x \in C$, $f(\lambda x)= \lambda f(x)$. Satisfying both these conditions is equivalent to the map being nonexpansive in the Funk metric.
In particular, we denote by $\End(C)$ the set of linear self-maps of $\R^n$ which preserve
the open cone $\Int C$. Any element of $\End(C)$ is nonexpansive for
the Funk metric.

For instance, if $C=(\R_{\geq 0})^n$, then,
$\End(C)=\mathcal{M}^n_{+}$, the set of nonnegative matrices $M$
with at least one non-zero entry per column, acting on row vectors
by $x\mapsto xM$.
If $C=H_n^+$, the set $\End(C)$ is harder to describe. It contains
all {\em completely positive maps} of the form
\(
T(x) = \sum_{i=1}^k a_i^*x a_i,
\)
where $a_1,\dots,a_k \in \C^{n\times n}$ are such that
$\sum_{i=1}^k a_i^* a_i$
is invertible.

The \emph{conical escape rate game} associated with the cone $C$
is defined as the escape rate game played on the interior of $C$,
equipped with the hemi-metric $d(x,y)=\Funk(x,y)$.
The two following examples are subclasses of such games.
\subsubsection{Families of Nonnegative Matrices}\label{Non-m}
We consider $C=(\R_{\geq 0})^n$. Let $\mathcal{A}$ and $\mathcal{B}$ 
be two non-empty compact subsets of $\mathcal{M}^n_+$.
Each pair of actions $(A,B)\in \mathcal{A}\times \mathcal{B}$ induces
a self-map of $(\R_{>0})^n$, given by
\[
T_{AB}(x) \coloneqq xAB\enspace,
\]
where $x$ is interpreted as a row vector.
Since $T_{AB}\in \End((\R_{\geq 0})^n)$, it is nonexpansive in the Funk hemi-metric.

This yields an escape rate game, in which the payoff is given by 
\begin{align*}
  \limsup_{k \to \infty} \frac{\Funk(x_k, x_0)}{k} &= \limsup_{k \to \infty} \log \max_{i \in \llbracket 1,n \rrbracket} (\frac{(x_k)_i}{(x_0)_i})^{\frac{1}k}\\
	& = \log (\limsup_{k \to \infty} (\max_{i \in \llbracket 1,n \rrbracket}(x_k)_i)^{1/k})\\
	& = \log(\limsup_{k \to \infty}\lVert x_k \rVert^{1/k}_{\infty})\enspace,
\end{align*}
where \[
x_k = x_0 A_1 B_1 \dots A_k B_k\enspace .
\]

Taking $x_0 = (1,\dots, 1)$ we get $\lVert x_k \rVert_{\infty} = \lVert A_1B_1 \dots A_kB_k \rVert$, where for $A = (a_{ij}) \in \R^{n\times n}$, $\lVert A \rVert = \max_{i}\sum_{k=1}^n \lvert a_{ik} \rvert$, and we recover (up to an exponential transformation) the payoff of a matrix multiplication game.

The {\em reverse Funk metric},
\[
\RFunk(x,y)\coloneqq\Funk(y,x) \enspace ,
\]
is also of interest. It yields an escape rate game in which the payoff is given by
\begin{align*}
  \limsup_{k\to\infty} \frac{\RFunk(x_k,x_0)}{k}
  & = \limsup_{k\to\infty}\log  \max_{i\in\llbracket 1,n\rrbracket}(\frac{(x_0)_i}{(x_k)_i})^{1/k}\\
  &= - \liminf_{k\to \infty} \log  \min_{i\in\llbracket 1,n\rrbracket}(\frac{(x_k)_i}{(x_0)_i})^{1/k}
\enspace .
\end{align*}
Then, one player wants to maximize the {\em smallest} growth rate of the coordinates of $x^k$,
whereas the other player wants to maximize it.

Such models arise in population dynamics, control of growth processes, and epidemiology: $x_k$ represents a population profile at time $k$; one player wishes to minimize the growth rate of the population whereas the other player wishes to maximize it, see~\cite{billy,calvezgabriel,Chakrabarti2008}.

\subsubsection{Families of Matrices in $\GL_n(\C)$} \label{Inv-m}
Every matrix $M\in \GL_n(\C)$ yields a \emph{congruence operator}
\[
	\phi_M : X \mapsto M^*XM \enspace ,
        \]
        which is an automorphism of the cone of positive semidefinite Hermitian matrices,
        i.e., $\phi_M\in \Aut(H_n^+)$.

        We say that a function $\nu$ on $\R^n$ is a {\em symmetric Gauge function} if it is
        invariant by permutation of the variables, convex and positively homogeneous of degree $1$. Every symmetric Gauge function $\nu$ yields a hemi-metric of Finsler type on $\Int(H_n^+)$ (\cite{nussbaumfinsler})
\[d_{\nu}(X,Y) = \nu(\log(\spec(XY^{-1}))) \enspace .\]
If $M\in \GL_n(\C)$, then,
we readily check that $\phi_M$ is an isometry of $(\Int (H_n^+), d_{\nu})$.

Consider $\mathcal{A}, \mathcal{B} \subset \GL_n$ and the operators $(\phi_{AB})_{(A,B) \in \mathcal{A}\times \mathcal{B}}$, which are
isometries, as $\phi_{AB} = \phi_B \circ \phi_A$  We denote by $\sing$ the set of singular values
of a matrix.
We get that
\[d_\nu(\phi_{A_kB_A}\circ \dots \circ \phi_{A_1B_1}(I),I)  =
2 \nu(\log(\sing(A_1B_1\dots A_kB_k)) \enspace .\]
Therefore, depending on the Gauge function we can get some insights about the spectrum and singular values of the infinite product of matrices. In particular, when $\nu(x)=\mytop(x)=\max_i x_i$,
\begin{align*}
d_{\mytop}(\phi_{A_kB_A}\circ \dots \circ \phi_{A_1B_1}(I),I)  &= \Funk(\phi_{A_kB_A}\circ \dots \circ \phi_{A_1B_1}(I),I)\\
&=2\log \|A_1B_1\dots A_kB_k\| \enspace ,
\end{align*}
where $\|\cdot\|$ is the spectral norm.
We recover the {\em matrix multiplication games} introduced
in~\cite{asarin_entropy}. 

\section{Preliminary Results}
To prove the existence and characterize the value of escape rate games, we use an operator approach (see \cite{Rosenberg2001}). 
Throughout the paper, we make the following assumption.
\begin{Assumption}\label{a-1}
For all $x\in X$,  the maps
$b \mapsto T_{ab}(x)$ and $a \mapsto T_{ab}(x)$ are continuous,
and for all compact sets $K$, the set $\{T_{ab}(x) \ | \ (a,b,x) \in \mathcal{A} \times \mathcal{B} \times K\}$ is compact.
\end{Assumption} 
This is trivially verified if $(a,b,x) \mapsto T_{ab}(x)$ is continuous. 

We denote by $\Lip$ the set of $1$-Lipschitz functions
from the hemi-metric space $(X,d)$ to $(\R,\delta_1)$ (recall that $\delta_1(x,y)= x-y$). We define the following {\em Shapley operator} $S$, acting on $\Lip$,
\begin{equation}
	Sv(x) \coloneqq \inf_{a\in \mathcal{A}} \sup_{b \in \mathcal{B}} v(T_{ab}(x)) \enspace .\label{e-def-shapley}
\end{equation}
The map $S$ is order preserving and commutes with the addition of a constant.
Moreover, as proven in \Cref{prop-contS} below, \Cref{a-1} ensures that $S$ preserves the set $\Lip$ and is continuous for the topology of uniform convergence on compact sets.
It also ensures that the infimum and supremum in~\eqref{e-def-shapley} are attained.

\begin{Proposition}
  Let $x_0 \in X$, the set $\mathcal{L}_{x_0}=\{f\in\Lip\mid f(x_0)=0\}$
  is compact for the topology of compact convergence (otherwise called uniform convergence on compact sets).
\end{Proposition}
	
	\begin{proof}
	  This set is obviously equi-continuous since it is equi-Lipschitz (for the hemi-metrics as well as the symmetrized metrics). Furthermore $\forall x\in X, \forall v \in \mathcal{L}_{x_0}$, we have $|v(x)| = |v(x) - v(x_0)| \leq d^{\circ}(x,x_0)$. So $\forall x$, $\{v(x) \mid v \in \mathcal{L}_{x_0} \}$ is a bounded set in $\R$ and therefore it is relatively compact. Now, by the  Ascoli-Arzela theorem,
we know that $\mathcal{L}_{x_0}$ is relatively compact for the topology of compact convergence (which coincides with the topology of pointwise convergence as $\mathcal{L}_{x_0}$ is equi-continuous \cite[Th 15 on p.232]{Kelley1975}). It is closed for this topology so it is compact.\end{proof}
\begin{Proposition} \label{prop-contS}
Under \Cref{a-1}, the operator $S$ preserves $\Lip$ and
is continuous from $\Lip$ to itself for the topology of compact convergence.
\end{Proposition}
\begin{proof}
By \Cref{a-1},
for any $x\in X$, and $a\in \A$, the map $b\in\B \mapsto T_{ab}(x)\in X$ 
is continuous.
Let $v\in \Lip$.
Since then $v$ is continuous from $X$ to $\R$ and $\B$ is compact,
we get that 
$Sv(x)\leq \sup_{b\in \B} v (T_{ab}(x))<+\infty$. Similarly, $Sv(x)>-\infty$.
This implies that $Sv(x)\in\R$ for all $x\in X$.
By assumption the maps $T_{ab}$ are nonexpansive, %
then all maps $x\mapsto v(T_{ab}(x))$ are in $\Lip$, so that $Sv$ also
belongs to $\Lip$, as $\Lip$ is preserved by suprema and infima.

Let us now show that $S$ is continuous for the topology of compact convergence.
	   Let $I$ be a directed set and consider a net $(v_i)_{i \in I}$ converging to a function $v$
          that belongs to $\Lip$. Let $K$ be a compact subset of $X$. 
For all $(x,a,b) \in K \times \mathcal{A} \times \mathcal{B}$, we have
\begin{align*}
v_i(T_{ab}(x)) &\leq v(T_{ab}(x)) + \sup_{(x,a,b) \in K \times \mathcal{A} \times \mathcal{B}} \lvert v_i(T_{ab}(x)) - v(T_{ab}(x)) \rvert \\
&= v(T_{ab}(x)) + \sup_{y \in Y} \lvert v_i(y) - v(y) \rvert \enspace ,
\end{align*}
where $Y\coloneqq\{T_{ab}(x) | (x,a,b) \in K \times \mathcal{A} \times \mathcal{B}\}$. So, for all $x\in K$, $Sv_i(x) \leq Sv(x) + \sup_{y \in Y} \lvert v_i(y) - v(y) \rvert$. By a symmetric argument and taking the supremum over $x\in K$, we get
\[\sup_{x \in K} |Sv_i(x)-Sv(x)| \leq \sup_{y \in Y  } |v_i(y) - v(y)| \enspace .
		\]
		By \Cref{a-1},
                $Y$ is compact.
                Therefore the RHS converges to 0 and so does the LHS. So $Sv_i \to Sv$ in the topology
                of compact convergence.
	\end{proof}
	
\begin{Proposition}\label{prop-nonempty}
	There exists at least one additive eigenvalue of $S$ associated to an eigenvector in $\Lip$, i.e.
	\begin{equation*}
		\exists \lambda \in \R \text{ such that } \exists \eigenvect \in \Lip, S\eigenvect = \lambda + \eigenvect \enspace .
	\end{equation*}
\end{Proposition}
	\begin{proof}
	  Let us recall that, given a topological space $X$,
          the space of continuous real-valued functions, $\mathcal{C}(X)$, equipped with the topology of uniform convergence on compact sets, which is defined by the family of semi-norms $\phi_K : f \mapsto \sup \{\lvert f(x) \rvert \mid x \in K \}$ where $K$ varies over the directed set of all compact subsets of $X$, is a locally convex topological vector space. Let us define $\tilde{S}\eigenvect \coloneqq S\eigenvect - S\eigenvect(x_0)$. It sends continuously $\mathcal{L}_{x_0}$ into itself. As $\mathcal{L}_{x_0}$, is a convex and compact subset of $\mathcal{C}(X)$, we can conclude that $\tilde{S}$ has a fixed point in $\mathcal{L}_{x_0}$ thanks to the generalization of the Schauder fixed point theorem by Tychonoff \cite{Tychonoff_1935}.
	\end{proof}

Note that this result becomes trivial if $d$ is a nonnegative hemi metric, since the null function would be an additive eigenvector with respect to the eigenvalue $0$. 

By the triangular inequality, the map $d(\cdot,x_0):x\mapsto d(x,x_0)$ is in $\Lip$. We define
\[
s_k\coloneqq [S^kd(\cdot,x_0)](x_0) \enspace.
\]
\begin{Proposition}\label{prop-subadd}
We have 
  \begin{align}
  \rho\coloneqq\inf_{k\geq 1} \frac{s_k}k  = \lim_k \frac{s_k}k \in \R
\enspace .\label{e-def-rho}
\end{align}
\end{Proposition}
\begin{proof}
  We first show the following subadditivity property:
  \begin{align}
    s_{k+l}\leq s_k+ s_l,\text{ for all } k,l\geq 1 \enspace .\label{e-subadd}
	\end{align}
Indeed, we have
	$ [S^l d(\cdot,x_0)](\cdot)\leq [S^l d(\cdot,x_0)](x_0)+ d(\cdot,x_0)=s_l + d(\cdot,x_0)$
	since $S^l d(\cdot,x_0)$ is $1$-Lipschitz. Then, since $S^k$
	is order preserving, and commutes with the addition of a constant,
	$S^{k+l}d(\cdot,x_0)
	= S^k(S^l d(\cdot,x_0))\leq S^k(s_l +d(\cdot,x_0))=s_l + S^k d(\cdot,x_0)$, and so, $[S^{k+l}d(\cdot,x_0)](x_0)\leq s_l+s_k$,
        showing \eqref{e-subadd}.
        Then, it follows from Fekete subadditive lemma 
        that the limit $\lim_k \frac{s_k}k$ does
        exists and that it coincides with $\rho\coloneqq \inf_k \frac{s_k}{k}$.
In particular $\rho\leq s_1<+\infty$. It remains to show that $\rho>-\infty$.

Let $v= - d(x_0, \cdot)$. By the triangular inequality, $v$ is 
in $\Lip$, so that $Sv(x_0)\in \R$. We set
\[
        \alpha\coloneqq[S (-d(x_0,\cdot))](x_0)=\inf_{a\in\mathcal{A}}\sup_{b\in\mathcal{B}} -d(x_0, T_{ab}(x_0))\in \R \enspace .
        \]
Again by the triangular inequality, we have 
$0=d(x_0,x_0)\leq d(x_0,\cdot)+d(\cdot,x_0)$.
We deduce that
        \[
        s_1 = [Sd(\cdot, x_0)](x_0)\geq [S (-d(x_0,\cdot))](x_0) =
        \alpha \enspace .
        \]
        We now show, by induction on $k$, that $s_k \geq \alpha k$ for all $k$.
        Since $S$ preserves $\Lip$, we have, for every $k \in \N^*$,
        $S^k d(\cdot,x_0)\in \Lip$, and so
\begin{equation*}
	[S^kd(\cdot,x_0)](\cdot) \geq 	[S^kd(\cdot,x_0)](x_0) - d(x_0,\cdot) = s_k - d(x_0, \cdot)  \enspace.
\end{equation*}
Since $S$ is order preserving and commutes with the addition of a constant,
we deduce 
\(
  S^{k+1}d(\cdot,x_0) \geq s_k +S(-d(x_0,\cdot))
  \).
Evaluating this inequality at point $x_0$, we get $s_{k+1} \geq s_k + \alpha$, which completes the induction.\end{proof}

The case in which there is only one player (Max) is of special interest. Then, the set $\A$ is trivial, being reduced to a single action $a$,
and we set $T_{b}\coloneqq T_{ab}$ for all $b\in \B$. We call such games
{\em minimizer-free}.
The following proposition shows that in this case, the value
of the escape rate game coincides with the {\em asymptotic joint displacement}
introduced by Breuillard and Fujiwara~\cite[Claim~4]{Breuillard2021}:

\begin{Proposition}
  The value of a minimizer-free escape rate game coincides with
  the following quantity:
\begin{align}
  \rho_{\operatorname{disp}} =
      \lim_{k\to\infty} \inf_{x} \sup_{b_1,\dots ,b_k \in \B}
  \frac{d(T_{b_k}\circ \dots\circ T_{b_1}(x),x)}{k} \enspace .
  \end{align}
  \end{Proposition}
\begin{proof}
In this special case, we have $s_k = \sup_{b_1,\dots ,b_k \in \B}
d(T_{b_k}\circ \dots\circ T_{b_1}(x_0),x_0)$, and $\rho=\lim_k s_k/k$. Moreover
it is observed in~\cite[Claim~4]{Breuillard2021} that for all choices
of $x_0$, $\lim_k s_k/k=\rho_{\operatorname{disp}}$.
\end{proof}
\begin{Remark}
  In our setting, at each turn, \Min\ plays before \Max, leading
  to a Shapley operator of the form $Sv(x)=\min_{a\in\A} \max_{b\in\B}
  v(T_{ab}(x))$. The variant in which \Max\ plays first can be tackled
  by considering instead the operator $S'v(x)=\max_{b}  \min_{a} v(T_{ab}(x))$. The results which follow carry over to this case, with trivial modifications.
  \end{Remark}
\section{Maximin and Minimax Characterizations of the Value of Escape Rate Games}
\subsection{Existence and Maximin characterization of the value}

  The following theorem characterizes the quantity $\rho$ defined by \eqref{e-def-rho}   in terms of the \emph{non-linear sub-eigenvalues} of $S$.

\begin{Theorem}
	\label{extremal}
	We have
\[ 
		\rho = \max\{\lambda \in \R \mid \exists \eigenvect \in \Lip, \; \lambda + \eigenvect \leq S\eigenvect \}\enspace .
\] 
\end{Theorem}

\begin{proof}
	We first observe that if $\lambda\in\R$ is such that $\exists \eigenvect \in \Lip$, $\lambda + v \leq Sv$, then, $k\lambda + v \leq S^kv$, and since $d(\cdot,x_0)\geq v -v(x_0)$,
	$S^kd(\cdot,x_0)\geq S^k (v-v(x_0))=S^kv-v(x_0)$,
	and so $s_k/k\geq ([S^kv](x_0)-v(x_0) )/k$, which entails that
	$\lim_k s_k/k \geq \lambda$. Therefore, $\rho \geq \sup\{\lambda \in \R \mid \exists \eigenvect \in \Lip, \; \lambda + \eigenvect \leq S\eigenvect \}$.

To complete the proof,
we need to construct a function $\eigenvect$ such that $S\eigenvect \geq \rho + \eigenvect$. Let $\lambda < \rho$. 
We first show the following result.
	
	\begin{Lemma}
		For any compact set $C \subset X$, $\exists k_C \in \N$ such that $S^{k_C}d(\cdot,x_0) \geq k_C\lambda + d(\cdot, x_0)$ on $C$.
	\end{Lemma}
	\begin{proof}
	Let $C$ be a compact subset of $X$. As $d(\cdot, x_0) \in \Lip$ and $S$ preserves $\Lip$, $\forall y \in C$, we have
	\begin{align*}
		[S^k d(\cdot, x_0)](y) &\geq [S^k d(\cdot, x_0)](x_0) - d(x_0,y)\\
&= s_k - d(x_0,y) \\ 
		&= \lambda k + d(y, x_0) + (s_k - \lambda k) - (d(y, x_0) + d(x_0,y)) \\
		&\geq \lambda k + d(y, x_0)  + (s_k - \lambda k) - \max_{x \in C}(d(x, x_0)  + d(x_0,x)) \enspace .
	\end{align*}
	Since $\rho - \lambda > 0$ and $\rho=\lim_{k\to \infty} s_k/k$,
there exists $k_C \in \N$ such that $(s_{k_C} - \lambda k_C) - \max_{x \in C}(d(x, x_0)  + d(x_0,x)) \geq 0$, which gives the result.\end{proof}
We now finish the proof of \Cref{extremal}. For convenience, we denote $d(\cdot, x_0)$ by $\eigenvectw$.
Let
\begin{align}
  \theta_C^\lambda \coloneqq \sup(\eigenvectw, S\eigenvectw - \lambda, \dots, S^{k_C -1}\eigenvectw - (k_C -1)\lambda) \enspace.
  \label{e-averaging}
\end{align}
As the supremum of nonexpansive functions from $X$ to $\R$, the function $\theta_C^\lambda$ is also nonexpansive. Moreover, for all $x\in C$ we have 
	\begin{align*}
		[S\theta_C^\lambda](x) &\geq \sup([S\eigenvectw](x),\dots,[S^{k_C} \eigenvectw](x) - (k_C-1)\lambda) \\
		&\geq \sup([S\eigenvectw](x),\dots,[S^{k_C -1}\eigenvectw](x) - (k_C -2)\lambda, \eigenvectw(x) + \lambda)= \lambda + \theta_C^\lambda(x)\enspace .
	\end{align*}
        Let $\mathcal{K}$ denote the set of all compact subsets of $X$,
        ordered by inclusion,
        and consider the net $\{C\}_{C\in \mathcal{K}}$.
After translating every function $\theta_C^\lambda$ by $-\theta_C^\lambda(x_0)$, we assume that $\theta_C^\lambda(x_0)=0$ holds for all $C\in \mathcal{K}$.
Then, the net $(\theta_C^\lambda)_{C\in \mathcal{K}}$ is included in $\mathcal{L}_{x_0}$,
which is compact. Hence, there is a subnet $(\theta_C^\lambda)_{C\in \mathcal{D}}$
along which $\theta_C^\lambda $ has a limit, denoted by $\theta^\lambda$,
and this limit belongs to $\mathcal{L}_{x_0}$.
Moreover, for all $x\in X$, 
the singleton $\{x\}$ belongs to $\mathcal{K}$, and then,
by definition of a subnet, there is an element $C_x\in \mathcal{D}$ such that $\{x\}\subset C_x$. It follows that $[S\theta_D^\lambda](x)\geq \lambda + \theta_D^\lambda(x)$ holds
for all $D\in \mathcal{D}$ such that $D\supset C_x$. Since the operator
$S$ is continuous,
taking the limit along the subnet, we deduce that
$[S\theta^\lambda](x) \geq \lambda +\theta^\lambda(x)$ holds for all $x\in X$.

Therefore, for all $\lambda < \rho$ we have constructed a function $\theta^\lambda \in \mathcal{L}_{x_0}$ such that
\[		S\theta^\lambda \geq \lambda + \theta^\lambda\enspace .\]	
Using again the compactness of $\mathcal{L}_{x_0}$, we deduce that the net $\{\theta^\lambda\}_{\lambda<\rho}$, indexed by the set of real numbers smaller than $\rho$, admits a subnet converging to a function $\theta\in \mathcal{L}_{x_0}$.
By continuity of $S$, this function verifies $S\theta \geq \rho + \theta$. Therefore, $\rho \leq \max \{\lambda \in \R \mid \exists v \in \Lip, \; \lambda + v \leq Sv\}$. This shows the equality in \Cref{extremal}.
\end{proof}

We will be interested with the existence of a {\em uniform value} as defined by Mertens and Neyman~\cite{mertens_neyman}, which reinforces the notion
of value of a game with mean-payoff objective.

\begin{Definition}[Uniform Value]\label{def-uniform}
	A zero-sum repeated game is said to have a \emph{uniform value} $v_{\infty}$ if both players can guarantee $v_{\infty}$ up to an arbitrary precision provided that they play for a sufficiently long time. Formally $v_{\infty}$ is the uniform value of the game if for all $\epsilon > 0$, there exist a pair of strategies $(\sigma_{\epsilon}, \tau_{\epsilon})$ and a time $N$ such that for all $k \geq N$ and for all pairs of strategies $(\sigma, \tau)$,
	\begin{align*}
		\frac{\payoff_k(\sigma_{\epsilon},\tau)}k \leq v_{\infty} + \epsilon \qquad\qquad
		\frac{\payoff_k(\sigma, \tau_{\epsilon})}k \geq v_{\infty}- \epsilon
                \enspace .
	\end{align*}
\end{Definition}
This implies that the non zero-sum variant of the game,
where \Min still wishes to minimize $J(\sigma, \tau)$ but \Max wishes to maximize \[
\payoff^-(\sigma,\tau) \coloneqq \liminf_{k\to \infty}\payoff_k(a_1b_1 \dots a_k b_k)/k \enspace ,\]
instead of the limsup, still has the value $v_\infty$, meaning that:
	\begin{align*}
		v_\infty \geq \inf_{\sigma}\sup_{\tau} \limsup_k\payoff_k(\sigma,\tau)/k \geq
		\sup_{\tau}\inf_{\sigma}  \liminf_k\payoff_k(\sigma,\tau)/k \geq
		v_\infty \enspace .
	\end{align*}
See  \cite{mertens_neyman} for more information on the uniform value; see also~\cite{bolte2013}.

We also introduce the notion of \emph{$q$-cyclic} strategy for player \Min.
In such a strategy, \Min chooses a strategy for the game $\Gamma_q$ and repeats it indefinitely. Formally, we construct it inductively: at every $mq + 1$ turn, $m\in \N$, player \Min forgets
the sequence of actions
and states played during the $mq$ first turns and then plays the game $\Gamma_q$ as if it were starting from $x_0$ again. This means
that for $m \in  \N, \ell = 1,\dots,q$, the sequence of actions and states \Min considers to choose their action at turn $mq + \ell$ is 
\begin{align*}
  (x_0, a_{mq + 1}, b_{mq + 1}, &T_{a_{mq + 1}b_{mq + 1}}(x_0), \dots \\
  & a_{mq + \ell - 1}, b_{mq + \ell - 1}, T_{a_{mq + \ell - 1}b_{mq + \ell - 1}}\circ \dots \circ T_{a_{mq + 1}b_{mq + 1}}(x_0))
\end{align*}
Note that the infinite
sequence of actions played by Min in such a $q$-cyclic strategy will not be periodic
unless player \Max also selects actions in a cyclic way. We now give, and prove, a more detailed statement of \Cref{main}

\begin{Theorem}\label{th-valueexists}
  The escape rate game has a uniform value, which coincides with $\rho$. Furthermore, given a
  sub-eigenvector $v$ associated to $\rho$, player \Max possesses a stationary optimal strategy obtained
  by maximizing $b \mapsto v(T_{a_{k+1}b}(x_k))$ at turn $k+1$.
Player \Min possesses a history dependent
    optimal strategy, and for all $\epsilon>0$, player \Min admits an $\epsilon$-optimal strategy that is $q$-cyclic for some $q$.
\end{Theorem}
\begin{proof}
  We first show the existence of the uniform value.

We know from \Cref{extremal} that there exists $v\in\Lip$ such that $\rho +v \leq Sv$.
Recall that any pair of strategies $(\sigma,\tau)$ of the two players \Min and \Max defines a sequence $(a_k,b_k)_{k\geq 0}$ of actions and that, starting from $x_0$,
the sequence of states of the game is defined inductively by
\( x_{\ell}= T_{a_\ell b_\ell}(x_{\ell-1})\), for $\ell\geq 1$.

	Let $\sigma$ be an arbitrary strategy of player \Min.
	If $a_1,b_1,\dots,a_{k-1},b_{k-1}$ is the sequence of actions played up to stage $k-1$,
the state $x_{k-1}$ of the game is given by 
 	\[ x_{k-1} \coloneqq T_{a_{k-1}b_{k-1}}\circ \dots \circ T_{a_1b_1}(x_0) \enspace,	\]
and $a_k$ is determined by $\sigma$ and the history.
	We then define the strategy $\tau^*$ of player \Max, as follows:
at step $k$, we select any action 
	\[
	b_{k} \in \argmax_{b \in \mathcal{B}} v(T_{a_{k}b}(x_{k - 1})) \enspace .
	\]
This is a stationary strategy as it depends only on the current state and the last action of \Min.
We get that the sequence of states generated by the strategies $\sigma$ and $\tau^*$
	satisfies 
        \begin{align*}
          Sv(x_{k-1})&= \min_{a\in \A}\max_{b\in \B} v(T_{ab}(x_{k - 1}))\\
&     
\leq \max_{b\in \B} v(T_{a_{k}b}(x_{k - 1}))=  v(T_{a_{k}b_k}(x_{k - 1}))= v(x_k)\enspace .
\end{align*}
	Since		$\rho + v \leq Sv$, 
	we deduce that $ \rho + v(x_{k-1})\leq v(x_k)$, for all $k\geq 1$,
and so
	\begin{align*}
		\rho \leq \frac{v(x_k) -v(x_0) }{k} \leq \frac{d(x_k,x_0)}{k}
\enspace .
	\end{align*}
        It follows that for all $k$, and for all strategies $\sigma$ of \Min,
        we have $\rho\leq J_k(\sigma,\tau^*)$, showing that the second inequality required by the notion of uniform value, as per Definition~\ref{def-uniform},
        is satisfied
        for any $\epsilon>0$, with a choice of $\tau_\epsilon=\tau^*$ independent
        of $\epsilon$ and with $N=1$.

Now, to show the first inequality in Definition~\ref{def-uniform}, we shall construct for all $\epsilon > 0$ a
$q$-cyclic $\epsilon$-optimal strategy for Min. Let $\epsilon > 0$, as $\rho = \inf_{k \in \N^*}\frac{s_k}{k}$, there exists $q \in \N^*$ such that $\frac{s_q}{q} \leq \rho + \epsilon/2$, the value of the game in horizon $q$ exists and is equal to $s_q$.  Let $\sigma_q$ be an optimal strategy of \Min for the game $\Gamma_q$. We define $\sigma_q^{\infty}$ as the $q$-cyclic strategy of \Min which consists of applying repeatedly $\sigma_q$. Let $\tau \in \strategyMax$.
As usual, we denote $(a_i, b_i)_{i \in \N^*}$ the sequence of actions induced by the strategies $\sigma_q^{\infty}$ and $\tau$ and for all $m \in \N$, we define the operator played from turn $(m-1)q + 1$ to turn $mq$, $\widetilde{T}_{w_m} \coloneqq T_{a_{mq}b_{mq}} \circ \dots \circ T_{a_{(m-1)q +1}b_{(m-1)q +1}}$. As $\sigma_q$ is an optimal strategy for the game $\Gamma_q$, we have for all $m \in \N$,
\[
	d(\widetilde{T}_{w_m}(x_0),x_0) \leq s_q \enspace .
\]
So using the properties of $d$ and the nonexpansiveness of the operators we get
	\begin{align*}
 \payoff_{mq}(\sigma_q^{\infty}, \tau) &= d(\widetilde{T}_{w_m}\circ \dots \circ \widetilde{T}_{w_1}(x_0),
x_0)\\
		&\leq d(\widetilde{T}_{w_m}\circ \dots \circ \widetilde{T}_{w_1}(x_0), \widetilde{T}_{w_m}(x_0)) 
+d(\widetilde{T}_{w_m}(x_0),x_0) \\
		&\leq  d(\widetilde{T}_{w_{m-1}}\circ \dots \circ \widetilde{T}_{w_1}(x_0), x_0) 
+d(\widetilde{T}_{w_m}(x_0),x_0)\\
		&
		  \qquad \qquad \qquad \qquad \qquad  \vdots\\
		&\leq d(\widetilde{T}_{w_m}(x_0),x_0) + \dots + d(\widetilde{T}_{w_1}(x_0),  x_0) \leq m s_q\leq mq(\rho + \epsilon/2)\enspace .
	\end{align*}
We set \begin{align}
  M\coloneqq\max_{(a,b)\in A\times B} d(T_{ab}(x_0),x_0)
  \enspace ,\label{e-def-M}
\end{align}
which is finite by \Cref{a-1}.

By the same arguments as above, we have for $0\leq \ell<q$, and $k=mq+\ell$,
\begin{align}
	\payoff_{k}(\sigma_q^{\infty}, \tau) &\leq
d(T_{a_{mq+\ell}b_{mq+\ell}} \circ \dots \circ T_{a_{mq +1}b_{mq +1}}(x_0),x_0) +
 m s_q\nonumber\\
&\leq d(T_{a_{mq+\ell}b_{mq+\ell}} (x_0),x_0)+\dots +d(T_{a_{mq +1}b_{mq +1}}(x_0),x_0) +
 m s_q\nonumber\\
 &\leq \ell M + m s_q \leq \ell M + mq(\rho+ \frac{\epsilon}2)\nonumber\\
 & = \ell (M-\rho - \frac{\epsilon}{2})  + k (\rho +\frac{\epsilon}2)
 \leq (q-1)(M^+-\rho^- -\frac{\epsilon}{2}) + k (\rho +\frac{\epsilon}2)
 \enspace ,\label{e-newb}
\end{align}
where $a^+\coloneqq\max(a,0)$ and $a^-\coloneqq\min(a,0)$.
Taking $N$ such that $(q-1)(M^+- \rho^- - \epsilon/2)/N\leq \epsilon/2$, we deduce
that for all $k\geq  \max(q,N)$, and for all strategies $\tau$ of \Max,
$J_k(\sigma_q^\infty,\tau)/k \leq \rho + \epsilon$, completing
the proof that $\rho$ is the uniform value of the game.
We also showed that \Max\ possesses a stationary optimal strategy,
whereas \Min\ possesses cyclic $\epsilon$-optimal strategies
for all $\epsilon>0$.

Now, to construct an optimal strategy for \Min, we follow the construction of a cyclic strategy but instead of keeping a fixed period, we increase it by $1$ every time.
This means that for the first turn \Min plays an optimal strategy associated with $\Gamma_{1}$. Then, he forgets the actions played so far (the other player can remember them) and plays an optimal strategy for the game $\Gamma_2$. Then, he forgets the first turn, and for the $2$ following turns, plays an optimal strategy associated with $\Gamma_{2}$. He repeats this operation, i.e., after the $1 + \dots + q$ first turns, he forgets the past and plays an optimal strategy associated with $\Gamma_{q+1}$. We denote by $\sigma^*$ this strategy of \Min.

By the same arguments as above (see the derivation of~\eqref{e-newb}),
we obtain 
for every $q \in \N^*$, $0\leq \ell\leq q$, and $\tau \in \strategyMax$ 
the following property:
	\begin{equation*}
		\payoff_{1 + \dots + q+\ell}(\sigma^*,\tau)
\leq \ell M + s_1 + \dots + s_q \enspace .
	\end{equation*}
Dividing by $1 + \dots + q+\ell$, we deduce, that for all $q\geq 1$,
	\begin{equation*}
		\frac{\payoff_{1 + \dots + q+\ell}(\sigma^*,\tau)}{1 + \dots + q+\ell} \leq \frac{q  M^+ + s_1 + \dots + s_q}{1 + \dots + q}
                = \frac{2 M^+}{q+1} + \frac{s_1 + \dots + s_q}{1+ \dots +q}
                \enspace .
	\end{equation*}
        Observe that any integer $k\geq 1$
        can be written in a unique way as $1+2+\dots +q +\ell $
        with $q\geq 1$ and $0\leq \ell \leq q$. We denote by $q_k$ the unique
        integer $q$ arising in this decomposition. 
        Then, for all $k \in \N^*$ and $\tau \in \strategyMax$,
        \[
        \frac{J_k(\sigma^*,\tau)}k \leq \frac{2M^+}{q_k + 1} + \frac{s_1 + \dots + s_{q_k}}{1 + \dots + q_k} \enspace .
        \]
        As $k$ goes to $+ \infty$, so does $q_k$.
        Moreover, since $\frac{s_k}k$ tends to $\rho$ as $k$ tends to infinity,
the weighted Cesaro sum $\frac{s_1+\dots+s_k}{1+2+\dots +k}$ tends to the same limit
(see Theorem~43, p100 of~\cite{hardy}).
        Then, for all $\epsilon>0$,
        we can find an integer $N$ such that for all $k\geq N$,  $\frac{s_1 + \dots + s_{q_k}}{1 + \dots + q_k}\leq \rho + \epsilon/2$, and $\frac{2M^+}{q_k + 1}\leq \epsilon/2$, so that $J_k(\sigma^*,\tau)/k \leq \rho + \epsilon$,
showing that $\sigma^*$ is an optimal policy of \Min\ in the
escape rate game.
\end{proof}

We have the direct corollaries for vector addition games, conical games, and matrix multiplication games.
\begin{Corollary}
	Vector addition games have a uniform value.
\end{Corollary}

\begin{Corollary}
	Conical escape rate games have a uniform value.
In particular, matrix multiplication games with matrices that are either invertible or preserve the interior of
  a proper cone
  have a uniform value.
\end{Corollary}

\subsection{Distance-like functions and minimax characterization of the value}

Recall that we say 
a $1$-Lipschitz function $v: X\to \R$ is {\em distance-like} if, for some
$x_0\in X$, there exists a constant $\alpha\in \R$ such that
\[
v(x) \geq \alpha+ d(x,x_0), \quad \forall x\in  X
 \enspace .
 \]
 Distance-like functions allow one to bound from above
the value of an escape rate game.
\begin{Proposition}
	\label{dlike}
  Suppose that there exists a distance-like function $v$ such that
  \[
  S v\leq \lambda + v \enspace ,
  \]
  for some $\lambda\in \R$. Then, the value $\rho$ of the escape rate game satisfies $\rho \leq \lambda$.
  Moreover, by selecting at state $x$ an action $a \in \argmin_{a \in \A}\sup_{b \in \B}v(T_{ab}(x))$, we obtain a stationary
  strategy for player \Min which guarantees $\lambda$.
\end{Proposition}
\begin{proof}
  By using this strategy, we get, by an immediate induction, that for any sequence of actions $(b_k)_{k \in \N^*}$ played by Max,
setting $  x_k = T_{a_kb_k} \circ\dots\circ T_{a_1 b_1}(x_0)$, 
  \begin{align*}
d(x_k,x_0) + \alpha &\leq   v(x_k) 
= v(T_{a_kb_k} (x_{k-1}))
\leq \sup_{b\in\B} v(T_{a_kb} (x_{k-1})) = Sv(x_{k-1})\\
&\leq v(x_{k-1}) + \lambda \leq\cdots  
\leq v(x_0) +k \lambda \enspace ,   \end{align*}
  showing that player \Min can guarantee $\lambda$. Hence, the value of the
  escape rate game satisfies $\rho\leq \lambda$.
  \end{proof}

Coupling \Cref{dlike} with \Cref{main} we get the following result.
\begin{Corollary}\label{prop-dlike}
	Suppose that there exists a distance-like function $v$ and $\lambda \in \R$ such that
	\[
	Sv = \lambda + v \enspace ,
	\]
	then, $\lambda$ is the value of the game. 
\end{Corollary}

We shall say that a nonexpansive map $T: X\to X$ is an {\em almost-isometry}
if for some $\gamma \geq 0$,
 \begin{align}\label{e-almost}
 d(T(x),T(y))\geq d(x,y) - \gamma \enspace , \forall x,y\in X \enspace .
 \end{align}
 In particular, an isometry is an almost-isometry (with $\gamma=0$). We shall
 say that the nonexpansive maps $T_{ab}$ with $(a,b)\in A\times B$ constitute
 a {\em uniform family of almost-isometries} if every map $T=T_{ab}$ satisfies \eqref{e-almost}
 for some constant $\gamma\geq 0$ independent of the choice of $(a,b)\in A\times B$.

 \begin{Proposition}\label{prop-preserve}
   Suppose that the maps $T_{ab}$ with $(a,b)\in A\times B$ constitute a uniform family
   of almost-isometries. Then, the Shapley operator $S$ preserve the space of distance-like
   functions $\mathscr{D}$.
 \end{Proposition}
 \begin{proof}
   Since $S$ commutes with the addition of a constant,
   it suffices to show that $S d(\cdot,x_0)$ is distance like.
   Then, defining $M$ as per~\eqref{e-def-M},
   we get that for all $(a,b)\in A\times B$ and $x \in X$,
   \[
   d(T_{ab}(x),x_0)\geq d(T_{ab}(x),T_{ab}(x_0))-d(x_0,T_{ab}(x_0)) \geq d(x,x_0)-\gamma -M \enspace. 
   \]
   This implies that $Sd(\cdot,x_0) \geq d(\cdot,x_0)-\gamma-M$. 
 \end{proof}

 We now prove the minimax characterization of the value of the game, stated
 as \Cref{th-dual} in the introduction.

\begin{proof}[Proof of \Cref{th-dual}]
  It follows from \Cref{dlike} that $\rho \leq \inf \{\lambda \in \R \mid v \in \mathscr{D}, Sv \leq \lambda + v \}	$.
  The proof of the reverse inequality relies on the following lemma.
	\begin{Lemma}\label{lemma-k}
		For all $\lambda > \rho$ there exists $k \in \N^*$ such that $S^kd(\cdot, x_0) \leq k\lambda + d(\cdot, x_0)$.
	\end{Lemma}
		\begin{proof}
			We have $S^kd(\cdot,x_0) \in \Lip$ so
			\[
				S^kd(\cdot,x_0) \leq [S^kd(\cdot,x_0)](x_0) + d(\cdot, x_0) = s_k +d(\cdot,x_0) \enspace .
			\]
			Since $\inf_{k \geq 1} s_k/k = \rho$ and $\lambda - \rho > 0$, there exists $k\geq 1$
                        such that $s_k/k \leq \lambda$. Then, $S d(\cdot,x_0)\leq k\lambda + d(\cdot,x_0)$.
		\end{proof}
                The end of the proof relies on a dual version of the ``sup-averaging'' argument \eqref{e-averaging}
                used in the proof of \Cref{extremal}. We take $k$ as in \eqref{lemma-k}, and define
                \[ \theta  \coloneqq \inf\{d(\cdot, x_0), Sd(\cdot,x_0)-\lambda, \dots, S^{k-1}d(\cdot, x_0) - (k -1)\lambda\}\enspace .\]
                Since, by \Cref{prop-preserve}, $S$ preserves the set of distance-like functions, and since this set is stable by taking pointwise infima of finite families of elements of this set,
                the function $\theta$ is distance-like. Moreover, since $S$ is order preserving, we have
		\begin{align*}
			S \theta &\leq  \inf\{Sd(\cdot, x_0), S^2d(\cdot,x_0) - \lambda, \dots, S^{k}d(\cdot, x_0) - (k-1)\lambda\} \\
			&\leq \inf\{Sd(\cdot, x_0), S^2d(\cdot,x_0)-\lambda, \dots, S^{k-1}d(\cdot, x_0) - (k-2)\lambda, d(\cdot, x_0) + \lambda\}\\
			& = \theta + \lambda \enspace . \qedhere
		\end{align*}	
\end{proof}
The infimum in \Cref{th-dual} is not reached in all generality, i.e.\ for some families of isometries there cannot exist a distance-like function such that $Sv \leq \rho + v$ as shown by the following example.

\begin{Example}
Take $X$ as being the Poincaré half space $\{z \in \C \mid \Im(z) > 0 \}$ equipped with the Poincaré Metric
	\[
	d(z,z') = 2\arcsinh(\frac{\lVert z - z' \rVert}{\sqrt{\Im(z)} \sqrt{\Im(z')} }) \enspace,
	\]
        where $\Im (\cdot)$ denotes the imaginary part of a complex number.
        We suppose the two players are dummies, so that
        a single dynamics is available $T(z) = z+1$.
It is trivial to see that $\lim_{k \to \infty}\frac{d(T^k(i),i)}{k} = 0$. Then, if there exists a distance-like function $v$ (of constant $\gamma \in \R$ and center $x_0$)  such that for all $x\in X$,  $v(Tx) \leq v(x)$, we would have for all $k \in \N^*$
\[ 
\gamma + 2\frac{\arcsinh(k)}{\Im(x_0)} \leq v(T^k(x_0)) \leq v(x_0) \enspace ,
\]
which is absurd as $\lim_{k \to \infty}\arcsinh(k) = + \infty$.

\end{Example}

Combining \Cref{dlike} and \Cref{th-dual}, we obtain:
\begin{Corollary}
 Suppose that the maps $T_{ab}$ with $(a,b)\in A\times B$ constitute a uniform family of almost-isometries.
Then, for all $\lambda>\rho$, player \Min has a stationary strategy which guarantees $\lambda$.\hfill\qed
  \end{Corollary}

\section{Conical escape rate games}\label{cone-games}
We now further investigate conical escape rate games,
introduced in Section~\ref{sec-funk}.
We consider a family of order-preserving and positively homogeneous
maps $(T_{ab})_{a\in\mathcal{A},b\in\mathcal{B}}$ acting on the interior of a proper cone $C \subset \R^n$.

Recall that every map $T_{ab}$ that is order preserving and positively homogeneous of degree $1$ is nonexpansive
on the space $X=\operatorname{Int} C$ equipped with the hemi-metric $d(x,y)=\Funk(x,y)$.  By Theorem~\ref{main}, the associated escape rate game
has a uniform value $\rho$, which admits a maximin characterization.
In contrast, the minimax characterization (\Cref{th-dual}) requires
a uniform family of almost isometries.
We next show that, owing to the particular conical setting, the almost-isometry
assumption can be dispensed with.

We consider a cross section of $C$ of the form $\Delta \coloneqq \{x \in C \mid \langle x,e^* \rangle = 1 \}$ where $e^*$ is a point in the interior of the dual cone $C^*$. We recall the following result.

\begin{Lemma}[Lemma~1.2.4 of~\cite{nussbaumlemmens}]\label{lem-compact}
	The cross section $\Delta$ of $C$ is compact for the Euclidean topology. 
\end{Lemma}
We fix an arbitrary norm $\|\cdot\|$ on $\R^n$.
\begin{Lemma}\label{lem-cst}
  There exist constants $\beta,\beta'\in \R$ such that for all $x\in C\setminus\{0\}$,
  \begin{align}\label{e-compare}
    -\beta  + \log \<x,e^*>  &\leq \Funk(x,x_0)  \leq \beta  + \log \<x,e^*> \enspace ,\\
  -\beta'  + \log \|x\|& \leq   \Funk(x,x_0)  \leq \beta'  + \log \|x\| \enspace .\label{e-compare2}
  \end{align}
\end{Lemma}
\begin{proof}
  Let $\Delta^*\coloneqq\{ \varphi\in C^*\mid \varphi(x_0)=1\}$. Applying Lemma~\ref{lem-compact}
  to the dual cone $C^*$, we get that $\Delta^*$ is compact. We have
  \[
  \Funk(x,x_0)= \log \max_{\varphi\in \Delta^*} \varphi(x) \enspace ,
  \]
  the maximum being achieved,
  see e.g.~\cite[Lemma~27]{GAUBERT_2011}.
  Since the map $(x,\varphi)\mapsto \varphi(x)$ is continuous, and $\Delta^*$
  is compact, it follows that $\Funk(\cdot,x_0)$ is continuous,
  and so, it is bounded on $\Delta$,
  so that there exists a constant $\beta\in\R$ such that
  $-\beta \leq \Funk(x,x_0) \leq \beta$ for all $x\in \Delta$. By homogeneity,
  we have $\Funk(\lambda x,x_0) = \Funk(x,x_0) + \log(\lambda)$ for all $\lambda>0$,
  and the same is true of $\log \<x,e^*>$, which shows that~\eqref{e-compare}
  holds. The proof of~\eqref{e-compare2} is similar.
\end{proof}
This lemma implies that the payment of an infinite play is the same if we replace
$\Funk(x_k,x_0)$ by $\log\|x_k\|$ or $\log\<x,e^*>$, for any choice of the norm $\|\cdot\|$
and for any choice of $e^*$ in the interior of $C^*$.

As the Thompson metric and the norm $\|\cdot\|$ define the same topology on the interior of a proper cone of $\R^n$ (see \cite{Nussbaum2012}), the operators $(T_{ab})_{(a,b) \in \mathcal{A}\times \mathcal{B}}$ are continuous for both topologies. In addition to \Cref{a-1}, we make the following assumption:

\begin{Assumption}\label{a-2}
  For all $(a,b) \in \mathcal{A}\times \mathcal{B}, T_{ab}$ can be continuously extended to $C$
  with respect to the norm topology of $C$.
  Furthermore, the map $\mathcal{A}\times \mathcal{B}\times C \to C, (a,b,x) \mapsto T_{ab}(x)$ is continuous,
$C$ being equipped with this topology.
\end{Assumption}
This assumption is trivially verified if the maps $T_{ab}$ are linear,
and in particular in Examples \ref{Non-m} and \ref{Inv-m}.
Note also that if the cone $C$ is polyhedral, any self-map of the interior
of $C$ that is order preserving and positively homogeneous
of degree $1$ can be extended continuously to $C$, see~\cite{MR1960046}.

We now reduce the conical escape rate game played on the state space $\Int C$ to a new game $\Gamma_{\Delta}$ played on the state space $\Delta$.
The two players \Min and \Max still alternatively choose actions in $\mathcal{A}$ and $\mathcal{B}$. The state of the game now evolves according to the \emph{normalized operators} $\hat{T}_{ab}(x)=\frac{T_{ab}(x)}{\langle T_{ab}(x),e^*\rangle}$. We suppose at each turn $k$, \Min\ gives to \Max\
a payment equal to $\log(\langle T_{a_{k}b_{k}}(\hat{x}_{k-1}), e^* \rangle)$, where $\hat{x_k} \coloneqq \hat{T}_{a_k b_k}\circ \dots \circ \hat{T}_{a_1 b_1}(x_0)$. So, the payoff of the game in horizon $k$ is 
\[
\sum_{k=1}^n\log(\langle T_{a_k b_k}(\hat{x}_{k-1}), e^* \rangle) = \log(\langle T_{a_n b_n} \circ \dots \circ T_{a_1 b_1}(x_0), e^* \rangle) \enspace .
\]
The payoff of an infinite play of the game $\Gamma_{\Delta}$ is given by
\[
\limsup_{k \to \infty}\frac{\log(\langle T_{a_k b_k} \circ \dots \circ T_{a_1 b_1}(x_0), e^* \rangle)}k \enspace .
\]
By~\eqref{e-compare}, this is identical to the payoff $\limsup_{k\to\infty}\Funk(x_k,x_0)/k$, so that the conical escape rate game is equivalent to the game $\Gamma_\Delta$ starting from any state $x_0$ in the relative interior of $\Delta$.

We now introduce $F$, the Shapley operator of $\Gamma_{\Delta}$, acting on the space of continuous
functions on $\Delta$, $\mathscr{C}(\Delta)$, as follows:
\[
Fv(x) \coloneqq \inf_{a \in \mathcal{A}}\sup_{b \in \mathcal{B}}
\Big\{\log(\langle T_{ab}(x), e^* \rangle) + v\big(\frac{T_{ab}(x)}{\langle T_{ab}(x),e^* \rangle}\big)\Big\} \enspace .
\]
We consider the subspace $\CLip_1(\Delta)$ of $\mathscr{C}(\Delta)$, consisting
of continuous functions $v: \Delta\to \R$ such that $v(x)-v(y)\leq \Funk(x,y)$
as soon as $\Funk(x,y)<\infty$.
Observe that if $T: C\to \R$ is continuous,
order preserving and positively homogeneous of degree $1$, then the restriction of $T$
to $\Delta$ belongs to $\CLip_1(\Delta)$.

\begin{Lemma}\label{preserves-clip}
	The map $F$ preserves $\CLip_1(\Delta)$.
\end{Lemma}

\begin{proof}
Let $v \in \CLip_1(\Delta)$. By assumption $(x,a,b) \mapsto (T_{ab}(x))$ is continuous, so is
\[
\phi : (x,a,b) \mapsto
\log(\langle T_{ab}(x), e^* \rangle) + v\big(\frac{T_{ab}(x)}{\langle T_{ab}(x), e^* \rangle}\big) \enspace .
\] 
If $f : X \times Y \to \R$ is a continuous function and $Y$ is compact then $x \mapsto \sup_{y \in Y}f(x,y)$, and $x \mapsto \inf_{y \in Y}f(x,y)$ are both continuous function.
Then as $\mathcal{B}$ is compact, $(x,a) \mapsto \sup_{b \in \mathcal{B}}\phi(x,a,b)$ is continuous and as $\mathcal{A}$ is compact, $Fv$ is continuous.

Let $x,y \in \Delta$ be such that $\Funk(x,y)<\infty$. Then, for all $(a,b) \in \A \times \B$, we have
\begin{align*}
&\phi(x,a,b)-\phi(y,a,b)\\ 
&= \log(\langle T_{ab}(x), e^* \rangle) - \log(\langle T_{ab}(y), e^* \rangle) + v\big(\frac{T_{ab}(x)}{\langle T_{ab}(x),e^* \rangle}\big) - v\big(\frac{T_{ab}(y)}{\langle T_{ab}(y),e^* \rangle}\big)\\
&\leq \log(\langle T_{ab}(x), e^* \rangle) - \log(\langle T_{ab}(y), e^* \rangle) + \Funk\big(\frac{T_{ab}(x)}{\langle T_{ab}(x),e^* \rangle},\frac{T_{ab}(y)}{\langle T_{ab}(y),e^* \rangle}\big)\\
&=\Funk(T_{ab}(x), T_{ab}(y))\leq \Funk(x,y)  \enspace .
\end{align*}
We deduce that $[Fv](x)-[Fv](y)\leq \Funk(x,y)$. 
\end{proof}

We now recall the following definition from~\cite{GAUBERT_2011}. It captures
a mild form of nonpositive curvature in the sense of Busemann.
\begin{Definition}
	We say that the hemi-metric space $(X,d)$ is \emph{metrically star-shaped} with center $x^*$ if there exists a family of geodesics $(\gamma_y)_{y \in X}$, such that $\gamma_y$ joins $x^*$ to $y$ and such that the following inequality is satisfied for every $(y,z) \in X^2$ and $s \in [0,1]$:
	\[
	d(\gamma_y(s), \gamma_z(s))\leq sd(y,z) \enspace .
	\]
        We also require that for every $y$, $\lim_{s \to 1}d(y,\gamma_y(s))=0$.
\end{Definition}
We now consider the space $X=\CLip_1(\Delta)$, equipped with the hemi-metric
$\delta : (f,g) \mapsto \max_{x \in \Delta}\big(f(x) - g(x)\big)$. We take the zero function as the center,
together with the choice of geodesic $\gamma_f$ from the zero function to a function $f\in \CLip_1(\Delta)$,
such that $\gamma_f(s) = s f$, for all $s\in [0,1]$. Then, the space $\CLip_1(\Delta)$
is metrically star shaped.
Moreover, the symmetrized metric of $\delta$ is the sup-norm
on $\CLip_1(\Delta)$.

\begin{Theorem}\label{DualCone}
  The value $\rho$ of a conical escape game satisfies
  \begin{align}
    \rho &= \lim_{k \to \infty}\frac{\delta(F^k(0),0)}k \label{e-add-1}\\
  &\ = \inf_{v \in \CLip_1(\Delta)}\max_{x \in \Delta}\big\{[F(v)](x) - v(x)\big\}\label{e-add-2}\\ 
& = \inf\{\lambda \in \R \mid \exists v \in \CLip_1(\Delta), Fv \leq \lambda + v \} \enspace .\label{e-add-3}
	\end{align}
\end{Theorem}
\begin{proof}
  We already observed that the existence of the (uniform) value $\rho$
  follows from Theorem~\ref{main}.

  Since the space $\CLip_1(\Delta)$ is metrically star shaped and complete,
  the first equality in 
  Theorem~\ref{GaubertVigeral} implies that the two expressions
  at the right-hand sides of~\eqref{e-add-1}--\eqref{e-add-2} are equal.
  The last equality~\eqref{e-add-3} is clear.
  Let $\bar{\rho}$ denote this common value.
  Let $\lambda>\bar{\rho}$, and let $v\in \CLip_1(\Delta)$ be such that $Fv\leq \lambda + v$. 
  Let us consider the stationary strategy $\sigma^*$ of \Min,
which in state $x\in C\setminus\{0\}$,
selects any action $a\in \mathcal{A}$ that achieves the minimum
in the expression of $[Fv](x)$. Then, applying any strategy $\tau$ of \Max,
we get a trajectory $x_0,x_1,\dots$ such that 
\begin{align*}
  \log \< x_k , e^*> + v(\frac{x_k}{\<x_k,e^*>})\leq k\lambda + v(\frac{x_0}{\<x_0,e^*>}) \enspace .
\end{align*}
By \Cref{lem-cst}, $\Funk(x_k,x_0)\leq \beta + \log\<x_k,e^*>$.
Since $v$ is bounded
by some constant $M$, it follows
that $\Funk(x_k,x_0) \leq \beta + M + k\lambda$,
which shows that \Min\ can guarantee any $\lambda>\bar{\rho}$.

We next show that $\rho\geq \bar{\rho}$,
using Corollary 23 from \cite{GAUBERT_2011}.
It shows that there exists $x^* \in \Delta$ (not necessarily in the interior
of $C$) such that for all $k \in \N$
\begin{align}\label{ineq-rhobar}
\frac{[F^k(0)](x^*)}{k}\geq \bar{\rho} \enspace .
\end{align}

Recall from~\Cref{prop-subadd} that $\rho = \lim_{k \to \infty} s_k/k$ where \[s_k=[S^k\Funk(\cdot,x_0)](x_0)\enspace,\]
so using equation \eqref{e-compare}, and the nonexpansive character of the operators $T_{ab}$,
we have 
\begin{align*}
&F^k(0)(x^*) \leq \beta + \min_{a_1}\max_{b_1} \dots \min_{a_k}\max_{b_k} \Funk(T_{a_k b_k}\circ \dots \circ T_{a_1b_1}(x^*), x_0)\\
&\leq \beta + \Funk(x^*, x_0) + \min_{a_1}\max_{b_1} \dots \min_{a_k}\max_{b_k} \Funk(T_{a_k b_k}\circ \dots \circ T_{a_1b_1}(x_0), x_0) \\
&= \beta + \Funk(x^*, x_0) + s_k \enspace .
\end{align*}
So, using~\eqref{ineq-rhobar}, 
\[
  \bar{\rho}\leq 
\frac{\beta + \Funk(x^*,x_0)}k  + \frac{s_k}k \enspace .
\]

As $x_0 \in \Int C$, $\Funk(x^*,x_0)$ is finite, so taking the limit as $k$ goes to infinity, we get $\bar{\rho}\leq \rho$.
So $\bar{\rho}$ coincides with the value $\rho$ of the escape rate game.
\end{proof}

Let us denote by $\mathscr{D}_C$ the space of continuous functions (for the euclidean topology) on $C$ which are distance-like (for the Funk metric) on the interior of $C$.
For a function $v \in \CLip_1(\Delta)$ we denote by $\bar{v}$ the \emph{lift} of $v$. This is a map defined on $C$ as follows, for all $(\lambda, x) \in \R_{>0}\times \Delta$, $\bar{v}(\lambda x)= \log(\lambda) +v(x)$. The map $v \mapsto \bar{v}$ is bijection between $\CLip_1(\Delta)$ and $\mathscr{D}_C$. If we define $S$ on $\mathscr{D}_C$, then one verifies easily that, for all $(v, \lambda) \in \CLip_1(\Delta) \times \R$, $Fv \leq \lambda +v$ is equivalent to $S\bar{v} \leq \lambda + \bar{v}$ on $C$. So we can rewrite the dual characterization of the value of the game as follows.

\begin{Corollary}\label{coro-cones}
		The value $\rho$ of escape rate games played on cones has the following dual characterization
		\[
		\rho = \inf\{\lambda \in \R \mid \exists v \in \mathscr{D}_C, Sv \leq \lambda + v \} \enspace .
		\]
\end{Corollary}
In this way, we obtain the minimax characterization without the assumption on almost-isometries. However, there is no guarantee that there is a function $v \in \mathscr{D}_C$ verifying $Sv \leq \rho + v$, or $Fv \leq \rho + v$ ---abusing notation and denoting by the same symbol $v$ a function on $\Delta$ and its lift.
Nonetheless, if a continuous eigenfunction exists, then the associated eigenvalue is equal to the value of the game:

\begin{Proposition}\label{coho}
If
\[
Fv = \lambda + v
\]
is verified for some $\lambda \in \R$ and $v \in \mathscr{C}(\Delta)$, then $\lambda$ is necessarily equal to $\rho$. 
\end{Proposition}

\begin{proof}
Indeed, in such a case the lifted function $v$ verifies $Sv = \lambda +v$ and we can construct stationary optimal strategies guaranteeing $\lambda$ in the game where the payoff is 
\[
\limsup_{k \to \infty}\frac{v(T_{a_k b_k}\circ \dots \circ T_{a_1b_1}(x_0))}{k} 
\] 
by letting \Min play an action $a_k \in \argmin_{a \in \mathcal{A}}\max_{b \in \mathcal{B}}v(T_{ab}(x_{k-1}))$ and \Max play an action $b_k \in$ $\argmax_{b \in \mathcal{B}}v(T_{a_k}b(x_{k-1}))$. 
Then as $v$ is continuous, it is bounded on $\Delta$ and because $v(\lambda x) = \lambda + v(x)$, for all $x \in \Delta$, we get that there exists $M \in \R^+$ such that for all $x \in X$,
\begin{equation}
\label{equivalentFunk}
\log(\langle x, e^* \rangle) - M \leq v(x) \leq \log(\langle x,e^* \rangle) + M \enspace .
\end{equation}
This inequality tells us this game is equivalent to $\Gamma_{\Delta}$ and therefore $\lambda = \rho$
\end{proof}

\begin{Remark}
  The eigenproblem above can be seen as a ``2-player cohomological equation''. In the classical cohomological equation~\cite{Livic1972}, one seeks a function $f$
satisfying
	\[
	g = f\circ T -f \enspace ,
	\]
	where $g$ is a given function on a state space $X$ equipped with a transition map $T: X \to X$.
\end{Remark}

We next observe that an eigenfunction does exist under the following
assumption.
\begin{Definition}[Tubular family]
  Let $K\subset \Int C$ be a cone such that the cross-section $\Delta_K\coloneqq \Delta \cap K$  is compact in the topology defined by the Thompson Metric.
A family $(T_{ab})$ of self-maps of the interior of a cone $C$ that are nonexpansive in the Funk Metric is {\em tubular} with respect to $K$
  if every map $T_{ab}$  preserves the cone $K$.
\end{Definition}
The terminology comes from the work of Bousch and Mairesse~\cite{Bousch_2001}, who dealt with the special case in which $C$ is the standard orthant
(up to a logarithmic change of variable).
Guglielmi and Protasov~\cite{Guglielmi2011} and Jungers~\cite{Jungers2012} used the same notion, under the name of {\em invariant embedded cone},
to study the lower spectral radius.
In the next proposition, we denote by $\CLip_1(\Delta_K)$ the set
of functions defined on $\Delta_K$ that are $1$-Lipschitz with respect to the
Funk metric induced by the cone $C$.
\begin{Proposition}\label{prop-tubular}
	If the family $(T_{ab})_{(a,b) \in \A \times \B}$ is tubular with respect to the cone $K$, then there exists $v \in \CLip_1(\Delta_K)$ such that 
	\[ Fv = \rho + v \]
	and the value of the escape rate game, $\rho$, is the unique eigenvalue of $F$.
\end{Proposition}
\begin{proof}
  
  We consider the escape rate game with state space restricted to $K$ (including its boundary),
  equipped with the restriction of the Funk metric of the cone $C$ to the cone $K$.
This game has the same value as the original conical game.
We saw in \Cref{prop-nonempty} that there exist $\lambda \in \R$ and a function $v$ defined on $\Int C$ that is $1$-Lipschitz with respect the Funk metric, and such that $Sv(x) = \lambda + v(x)$ holds for all $x\in  \Int C$. As $\Delta_K$ is compact, the function $v$ is bounded on $\Delta_K$, in particular,
$m\coloneqq \min_{x\in\Delta_K}v(x)>-\infty$.
By Lemma~\ref{lem-cst}, 
this implies that $v(x)=v((\<x,e^*>)^{-1}x) + \log \<x,e^*> \geq \Funk(x,x_0)-\beta+m$,
and so $v$ is distance-like on $K$.
Moreover, we have
  $[Fv](x) = \lambda + v(x)$ for all $x\in\Delta_
  K$. 
We get from \Cref{prop-dlike} that $\lambda = \rho$ which proves the proposition. 
\end{proof}

To end this section, we point out that the existence of a distance-like function $v$ such that $Sv \leq \rho + v$ is analogous to the existence of an \emph{extremal norm} for a set of matrices $\mathcal{A}$, i.e. a submultiplicative norm $N$ such that $\sup_{A \in \mathcal{A}}N(A) = \rho(\mathcal{A})$, where $\rho(\cdot)$ denotes the joint spectral radius. The existence of such a norm is equivalent to the \emph{non-defectivity} of $\mathcal{A}$, see~\cite{Barabanov1988,Kozyakin1990,Jungers_2009}. Recall that the set of matrices $\mathcal{A}$ is \emph{non-defective} if, for an arbitrary norm $\|\cdot\|$,  there exists a positive constant $C$ such that for all $k \in \N, \sup_{A \in \mathcal{A}^k}\lVert A \rVert \leq C\rho^k$.
We next show that a similar property holds for escape rate games.
\begin{Proposition}
	\label{non-defective}
	In the minimizer-free escape rate game the two following properties are equivalent
	\begin{enumerate}
		\item $\exists \alpha \in \R$ such that $\forall k \in \N^*$, $S^kd(\cdot, x_0) \leq k\rho + d(\cdot, x_0) + \alpha$
		\item $\exists v \in \mathscr{D}$ such that $Sv \leq \rho + v$
	\end{enumerate}
\end{Proposition}
\begin{proof}
	$(1) \implies (2)$. If (1) is verified we can define $v \coloneqq \sup_{k \in \N}S^kd(\cdot, x_0)-k\rho$ which is distance-like as it is greater than $d(\cdot,x_0)$. Since the game is assumed to be minimizer-free, the operator $S$ commutes with arbitrary suprema, and so
	\[
	Sv = \sup_{k \in \N}S^{k+1}d(\cdot,x_0)-k\rho = \sup_{k \in \N^*}S^k d(\cdot, x_0)-k\rho + \rho \leq v + \rho \enspace .
	\]
	$(2) \implies (1)$. If (2) is verified, let $v \in \mathscr{D}$ such that $Sv \leq \rho + v$. Then there exists $\alpha \in \R$ such that for all $k \in \N^*$,
	\begin{equation*}
	S^kd(\cdot, x_0) + \alpha \leq S^k v \leq k\rho + v \leq k \rho + d(\cdot, x_0) + v(x_0) \enspace . \qedhere
\end{equation*}\end{proof}
This result suggests that for escape rate games, the distance-like functions that achieve the infimum in \Cref{coro-cones}
play the role of extremal norms.

\section{The Example of Vector Addition Games}\label{sec-vecadd}
As an illustration, we give a complete solution
of the vector addition game.
We denote by $((\R^n)^*, \|\cdot\|_*)$ the dual normed space of $(\R^n,\|\cdot\|)$.
Let us first notice that this game is trivial if $-\mathcal{B} \subset \mathcal{A}$. Then, for every $k \in \N^*$, \Min\ can play tit-for-tat, setting $a_{k+1}\coloneqq -b_k$ at turn $k+1$, canceling the previous translation.
This leads to a bounded trajectory,
implying that the value of the escape rate is zero.
We next denote by $\co(\cdot)$ the convex hull of a set.
\begin{Proposition}
    The vector addition game has a value equal to 
    \begin{equation}
        \lambda \coloneqq \max_{b \in \mathcal{B}} \min_{a \in \co(\mathcal{A})} \lVert a + b \rVert \enspace .
    \end{equation}
Any linear form $\ell$ achieving the maximum
	in 
	\[
  \max_{\ell \in (\R^n)^*, \| \ell \|_* \leq 1} \max_{b \in \mathcal{B}} \min_{a \in \co(\mathcal{A})} \ell(a + b) 
  \]is a solution of the eigenproblem $S \ell = \lambda + \ell$.
Furthermore, Max has a constant optimal strategy by playing
    \begin{equation}
        b^* \in \argmax_{b \in \mathcal{B}} \min_{a \in \co(\mathcal{A})}\lVert a + b \rVert \enspace .
    \end{equation}
Finally, the function
  \[
  \phi(x) = \text{dist}(x, -(n+1)\co(\mathcal{A}))
  \]
is distance-like and verifies $S\phi \leq \lambda + \phi$. Hence, by playing at state $x\in \R^n$ an action $a$ which achieves the minimum in the expression of $S\phi(x)$, we obtain
a stationary optimal strategy of \Min.
\end{Proposition}

\begin{proof}
  Using the characterization of the norm as a supremum over linear forms,
  together with Sion theorem on the commutation of min and max~\cite{Sion1958}, we get:
		\begin{align*}
			\lambda &= \max_{b \in \mathcal{B}} \min_{a \in \co(\mathcal{A})} \lVert a + b \rVert = \max_{b \in \mathcal{B}} \min_{a \in \co(\mathcal{A})} \max_{\ell \in (\R^n)^*, \| \ell \|_* \leq 1}\ell(a+b)\\
			&= \max_{b \in \mathcal{B}}  \max_{\ell \in (\R^n)^*, \| \ell \|_* \leq 1}\min_{a \in \co(\mathcal{A})}\ell(a+b)\ = \max_{\ell \in (\R^n)^*, \| \ell \|_* \leq 1} \max_{b \in \mathcal{B}} \min_{a \in \co(\mathcal{A})}  \ell(a + b) \\
                        & = \max_{\ell \in (\R^n)^*, \| \ell \|_* \leq 1}\max_{b \in \mathcal{B}} \min_{a\in \mathcal{A}}  \ell(a + b)
                        = \max_{\ell \in (\R^n)^*, \| \ell \|_* \leq 1}(\max_{b\in \mathcal{B}} \ell(b) + \min_{a\in \mathcal{A}} \ell(a))\\
                        & =
                          \max_{\ell \in (\R^n)^*, \| \ell \|_* \leq 1} \min_{a\in \mathcal{A}} \max_{b\in \mathcal{B}} \ell(a+b) \enspace .
		\end{align*}
		Let $\ell$ be a linear form achieving the maximum in the expression above. Then, by definition of $S$, and using the linearity of $\ell$, we get
		\begin{align*}
		  S\ell &= \ell + \min_{a\in \mathcal{A}}\max_{b \in \mathcal{B}}\ell(a+b) = \ell + \lambda \enspace .
		\end{align*}
        Let \Max play $b^* \in \argmax_{b \in \mathcal{B}} \min_{a \in \co(\mathcal{A})}\|a+b\|$ then for any $k \in \N^*$ we have for all $(a_1,\dots, a_k) \in \mathcal{A}^k$,
        \begin{equation*}
            \frac{\lVert x_0 + a_1 + b^* + \dots + a_k + b^* - x_0 \rVert}{k} = \lVert \frac{a_1 + \dots + a_k}{k} + b^* \rVert \geq \max_{b \in \mathcal{B}} \min_{a \in \co(\mathcal{A})} \lVert a + b \rVert \enspace ,
        \end{equation*}
        since $\frac{a_1 + \dots a_k}{k} \in \co(\mathcal{A})$.
    
For the second inequality, let us first recall the useful \emph{Shapley Folkman lemma}. Here, $\sum_{k=1}^N X_k$ denotes the Minkowski sum of subsets $X_1,\dots, X_N\subset \R^n$.

    \begin{Lemma}[Shapley Folkman \cite{Starr1969}]\label{lem-shapley}
        Let $N \in \N^*$ such that $N > n$. If $X_1,\dots, X_N$ are nonempty bounded subset of $\R^n$ then for any $x \in \co(\sum_{k=1}^N X_k)= \sum_{k=1}^N\co(X_k)$, there exists $I \subset \{1,\dots, N\}$ such that $|I| \leq n$ and 
        \begin{equation*}
            x = \sum_{i \in I}q_i + \sum_{k \in \{1, \dots , N \}/I}q_k \enspace ,
        \end{equation*}
        where $q_i \in \co(X_i) / X_i$ and $q_k \in X_k$
    \end{Lemma}
    Let $\tau \in \strategyMax$ a strategy of \Max. To simplify the notations, we are going to denote $\max_{b \in \mathcal{B}} \min_{a \in \co(\mathcal{A})} \lVert a + b \rVert = \max_{b \in \mathcal{B}}d(b, -\co(\mathcal{A}))$ by $\lambda$. Given a strategy $\sigma$ of \Min we will denote the resulting position after $k$ turns, i.e $x_0 + a_1 + b_1 + \dots + a_k + b_k$ by $x_k(\sigma, \tau)$. We are going to prove that if $x_0 \in -n\co(\mathcal{A})$, there exists a strategy $\sigma^* \in \strategyMin$ such that for all $k \in \N$
    \begin{equation*}
        \delta_k \coloneqq d(x_k(\sigma^*, \tau), -n\co(\mathcal{A})) \leq k\lambda \enspace .
    \end{equation*}
    
   This would mean that, starting from $-n\co(\mathcal{A})$ we can guarantee a payoff lower than $\lambda$ and, as the initial position does not matter, we would have proven our proposition. Given the preceding sentence, we can consider the game where \Max plays first. 

   Let us assume that for a certain $k \in \N$, the above inequality is true. It means that $x_k \in -n\co(\mathcal{A}) + kB(\lambda)$,
   where $B(\lambda)$ denotes the closed ball centered at $0$ of radius $\lambda$,
and since, by definition, the move $b_{k+1}$ played by \Max given the strategy $\tau$ is in $-\co(\mathcal{A}) + B(\lambda)$. So 
    \begin{equation*}
        x_k + b_{k+1} \in -(n+1)\co(\mathcal{A}) + (k+1)B(\lambda)
    \end{equation*}
    The Shapley-Folkman lemma tells us that $-(n+1)\co(\mathcal{A}) \subset -n\co(\mathcal{A}) - \mathcal{A}$. So there exists at least one $\hat{a} \in \mathcal{A}$ such that 
    \begin{equation*}
        x_k + b_{k+1} \in -n\co(\mathcal{A}) - \hat{a} + (k+1)B(\lambda) \enspace .
    \end{equation*}
    So by playing $a_k+1 = \hat{a}$, we get that 
    \begin{equation*}
        x_{k+1} \in -n\co(\mathcal{A}) + (k+1)B(\lambda) \enspace .
    \end{equation*}
    Then our proposition follows by induction.
    Now, this setting not only helps to better understand what escape rate games are, but it illustrates \Cref{dlike} and the fact that the distance-like function $\phi$ gives us the reverse inequality immediately.
    Indeed
	\begin{align*}
	  S\phi(x) &= \min_{a \in \mathcal{A}}\max_{b \in \mathcal{B}}d(x+a+b, -(n+1)\co(\mathcal{A}))  \\          &= \min_{a \in \mathcal{A}}\max_{b \in \mathcal{B}}\min_{c \in (n+1)\co(\mathcal{A})} \lVert x+a+b + c \rVert  \\
		&\leq \min_{a \in \mathcal{A}}\min_{a' \in n\co(\mathcal{A})}\max_{b \in \mathcal{B}}\min_{d \in \co(\mathcal{A})}\lVert x + a + a' + b + d \rVert \\
		&\leq \min_{a \in \mathcal{A}}\min_{a' \in n\co(\mathcal{A})}\max_{b \in \mathcal{B}}\min_{d \in \co(\mathcal{A})}\lVert x + a + a'\rVert + \lVert  b + d \rVert \\
		&= \min_{c \in (n+1)\co(\mathcal{A})}\max_{b \in \mathcal{B}}\min_{d \in \co(\mathcal{A})}\lVert x + c\rVert + \lVert  b + d \rVert \quad \text{ (by Lemma~\ref{lem-shapley})}\\
		&= \phi(x) + \max_{b \in \mathcal{B}}\min_{d \in \co(\mathcal{A})}\lVert b + d \rVert  = \phi(x) +\lambda \enspace . \qedhere
		\end{align*}
\end{proof}
         \section{Concluding remarks}
         We finally raise some open questions.
         We observed that when the non-linear eigenproblem $Sv=\lambda +v$ is solvable with a function $v$ that is distance like, then, $\lambda$ is unique and coincides with the value of the escape rate game.
         It would interesting to give conditions for the solvability of this eigenproblem,
         beyond the tubular setting of Proposition~\ref{prop-tubular};
we note that related issues have arisen in the study of the cohomological equation and its generalizations, in the context of dynamical systems~\cite{Livic1972,Bousch2001,Bousch2011}.

Whereas, in the escape rate game, Player \Max\ always has an optimal stationary strategy,
we only showed that Player \Min\ has an optimal stationary strategy
when the non-linear eigenproblem $Sv\leq \rho +v$ is solvable for some distance-like
function $v$. Investigating the existence of an optimal stationary strategy
of Player \Min\ beyond this case remains an open question.

Our final question concerns matrix multiplication games. We dealt with
payoff functions of the form $\limsup_k k^{-1}\log \|A_1B_1\dots A_kB_k\|$ where $\|\cdot\|$ is a norm
on the matrix space. Dealing with pointwise payoffs of the form
$\limsup_k k^{-1}\log \|xA_1B_1\dots A_kB_k\|$,
where now $\|\cdot\|$ is a norm on vectors, and $x$ is an initial vector, would be interesting.
This payoff generally depends on the choice of $x$.
A special case concerns finite sets of nonnegative matrices satisfying
a rectangularity assumption; this is known as entropy games~\cite{asarin_entropy}. The pointwise
version was solved in~\cite{Akian2019} using o-minimal geometry techniques~\cite{bolte2013}.

{ \hypersetup{allcolors=black}
\printbibliography
}
\end{document}